\documentclass[12pt,a4paper]{amsart}

\usepackage{accents}

\usepackage{mathrsfs}
\usepackage{todonotes}
\usepackage[utf8]{inputenc}

\usepackage{amssymb, tikz}

\usepackage{hyperref}

\usepackage{amsmath}

\newtheorem{theorem}{Theorem}[section]
\newtheorem{claim}[theorem]{Claim}

\newtheorem{lemma}[theorem]{Lemma}
\newtheorem{proposition}[theorem]{Proposition}

\newtheorem{corollary}[theorem]{Corollary}

\theoremstyle{definition}
\newtheorem{definition}[theorem]{Definition}

\theoremstyle{remark}
\newtheorem{remark}[theorem]{Remark}
\newtheorem{question}[theorem]{Question}

\DeclareMathOperator{\GCH}{\textsf{GCH}}
\DeclareMathOperator{\CH}{\textsf{CH}}

\DeclareMathOperator{\PFA}{\textsf{PFA}}
\DeclareMathOperator{\BPFA}{\textsf{BPFA}}
\DeclareMathOperator{\MM}{\textsf{MM}}
\DeclareMathOperator{\ZFC}{\textsf{ZFC}}
\DeclareMathOperator{\MRP}{\textsf{MRP}}
\DeclareMathOperator{\FA}{\textsf{FA}}

\DeclareMathOperator{\TC}{TC}
\DeclareMathOperator{\Club}{Club}
\DeclareMathOperator{\Id}{Id}
\DeclareMathOperator{\OCA}{\textsf{OCA}}
\DeclareMathOperator{\Res}{Res}
\DeclareMathOperator{\NS}{NS}
\DeclareMathOperator{\BA}{\textsf{BA}}
\DeclareMathOperator{\MA}{\textsf{MA}}
\DeclareMathOperator{\WCG}{\textsf{WCG}}
\DeclareMathOperator{\TOP}{\textsf{TOP}}
\DeclareMathOperator{\Fml}{Fml}
\DeclareMathOperator{\Add}{Add}

\newcommand{\mtcl}{\mathcal}
\newcommand{\sub}{\subseteq}

\newcommand{\Lim}{{\rm Lim}}

\newcommand{\dom}{{\rm dom}}

\newcommand{\cH}{{\mathscr H}}

\newcommand{\bbP}{{\mathbb P}}

\newcommand{\bbQ}{{\mathbb Q}}

\newcommand{\cf}{{\rm cf}}

\newcount\skewfactor
\def\mathunderaccent#1#2 {\let\theaccent#1\skewfactor#2
\mathpalette\putaccentunder}
\def\putaccentunder#1#2{\oalign{$#1#2$\crcr\hidewidth
\vbox to.2ex{\hbox{$#1\skew\skewfactor\theaccent{}$}\vss}\hidewidth}}
\def\name{\mathunderaccent\tilde-3 }

\begin{document}

\title[$\PFA$ for $\aleph_1$-sized posets and the size of the continuum]{The proper forcing axiom for $\aleph_1$-sized posets, $\omega_1$-linked symmetrically proper forcing, and the size of the continuum}

\author[D.\ Asper\'o]{David Asper\'{o}}

\address{David Asper\'o, School of Mathematics, University of East Anglia, Norwich NR4 7TJ, UK}

\email{d.aspero@uea.ac.uk}

\author[M.\ Golshani]{Mohammad Golshani}
\address{School of Mathematics\\
 Institute for Research in Fundamental Sciences (IPM)\\
  P.O. Box:
19395-5746\\
 Tehran-Iran.}
\email{golshani.m@gmail.com}
\urladdr{http://math.ipm.ac.ir/~golshani/}

\thanks{The second author's research has been supported by a grant from IPM (No. 1401030417). }



\begin{abstract}
We show that the Proper Forcing Axiom for forcing notions of size $\aleph_1$ is consistent with the continuum being arbitrarily large. In fact, assuming $\GCH$ holds and $\kappa\geq\omega_2$ is a regular cardinal, we prove that there is a proper and $\aleph_2$-c.c.\ forcing giving rise to a model of this forcing axiom together with $2^{\aleph_0}=\kappa$ and  which, in addition, satisfies all statements of the form $\cH(\aleph_2)\models \exists y\varphi(a, y)$, where $a\in \cH(\aleph_2)$ and $\varphi(x, y)$ is a $\Sigma_0$ formula with the property that for every ground model $M$ of $\CH$ with $a\in M$ there is, in $M$, a suitably nice poset---specifically, a poset $\bbQ\subseteq\cH(\kappa)^M$ which is $\omega_1$-linked and symmetrically proper---adding some $b$ such that $\varphi(a, b)$. In particular, $\bbP$ forces Moore's Measuring principle, Baumgartner's Axiom for $\aleph_1$-dense sets of reals, Todor\v{c}evi\'{c}'s Open Colouring Axiom for sets of size $\aleph_1$, the Abraham-Rubin-Shelah Open Colouring Axiom, and Todor\v{c}evi\'{c}'s  P-ideal Dichotomy for $\aleph_1$-generated ideals on $\omega_1$, among other statements. Hence, all these statements are simultaneously compatible with a large continuum. Finally, we show that a further small variation of our construction yields a model satisfying, in addition to all the earlier conclusions, Martin's Maximum for posets of size $\aleph_1$.
\end{abstract}

\maketitle
\setcounter{section}{-1}

\section{Introduction}
Forcing axioms can be considered as generalizations of the Baire category theorem and spell out one version of the idea that the universe of sets should be rich. The first example of a forcing axiom, introduced by Martin (see \cite{solovay}), is known as Martin's Axiom. In this paper we concentrate on a generalization of Martin's Axiom known as the Proper Forcing Axiom ($\PFA$).
$\PFA$ was introduced by Baumgartner \cite{baumgartner} (and Shelah \cite{Sh:b}) and states that given a proper forcing notion $\bbP$ and a collection $\mathcal D$
of $\aleph_1$-many dense subsets of $\bbP$ there exists a filter $G \subseteq \bbP$ meeting all the members of $\mathcal D$. $\PFA$ is consistent modulo
the existence of a supercompact cardinal and has many consequences for the structure of the universe; in particular, by the work of
Todor\v{c}evi\'{c} and Veli\v{c}kovi\'{c}  (see for example \cite{bekkali} and \cite{velickovic}), it implies that the size of the continuum is $\aleph_2$.

Throughout this paper let us denote by $\PFA(\omega_1)$ the restriction of $\PFA$ to posets of cardinality at most $\aleph_1$; i.e., $\PFA(\omega_1)$ is the statement that if $\bbP$ is a proper forcing notion such that $|\bbP|\leq\aleph_1$ and $\mathcal D$ is a collection of $\aleph_1$-many dense subsets of $\bbP$, then there is a filter $G\subseteq\bbP$ meeting all members of $\mathcal D$.

$\PFA(\omega_1)$ clearly implies Martin's Axiom for collections of $\aleph_1$-many dense sets (this is usually denoted by $\MA_{\aleph_1}$). It properly extends $\MA_{\aleph_1}$ as various non-c.c.c.\ proper forcings of size $\aleph_1$ fall in its range---for example Baumgartner's forcing for adding a club of $\omega_1$ with finite conditions, or natural forcings for adding, by finite conditions, various instances of the negation of Club Guessing at $\omega_1$. Thus, it is easy to see, for example, that $\PFA(\omega_1)$ implies $\lnot\WCG$, where $\WCG$ denotes weak Club Guessing.\footnote{Weak Club Guessing is the statement that there is a ladder system $\langle C_\delta\,:\,\delta\in\Lim(\omega_1)\rangle$ (i.e., each $C_\delta$ is a cofinal subset of $\delta$ of order type $\omega$) such that every club of $\omega_1$ has infinite intersection with some $C_\delta$. It is easy to see that $\MA_{\aleph_1}$ is compatible with $\WCG$ since $\WCG$ is preserved by c.c.c.\ forcing.} It is also worth mentioning that $\PFA(\omega_1)$ implies that every two normal Aronszajn trees $T$ and $U$ are club-isomorphic, i.e., there is a club $C\subseteq\omega_1$ such that the subtrees $T\restriction C=\bigcup_{\alpha\in C}\{t\in T\,:\, \text{ht}_T(t)=\alpha\}$ and $U\restriction C=\bigcup_{\alpha\in C}\{u\in U\,:\, \text{ht}_U(u)=\alpha\}$ are isomorphic (s.\ \cite{tod-handbook}, Theorem 5.10).

Shelah \cite{Sh:b} showed that the consistency of $\PFA(\omega_1)$ does not need any large cardinal hypotheses. In fact, starting with a model of $\GCH$, one can easily force $\PFA(\omega_1)$ by means of a suitable countable support iteration of proper forcings of size $\aleph_1$. In this model $2^{\aleph_0}=\aleph_2$ holds.  The question whether $\PFA(\omega_1)$ decides the value of the continuum remained an open problem.

There is a wide range of works showing the consistency of forcing axioms, or of consequences of forcing axioms, with the continuum being larger than $\aleph_2$;
 see for example \cite{aspero-mota1}, \cite{aspero-mota2}, \cite{cox}, \cite{gilton1}, \cite{gilton2},
  \cite{mohammadpour}, \cite{mohammadpour1} and \cite{sh1988}. For instance, and most to the point for us in the present paper, it is shown in \cite{aspero-mota1} using forcing with side conditions that  $\PFA$ restricted to certain classes of posets with the $\aleph_2$-chain condition is consistent with $2^{\aleph_0}>\aleph_2$.\footnote{Some form of side condition is usually necessary in the constructions we are referring to. At any rate, and as is well-known, any countable support iteration $\langle \bbP_\alpha\,:\,\alpha\leq\lambda\rangle$ of nontrivial forcing notions will collapse $(2^{\aleph_0})^{\bold V^{\bbP_\alpha}}$ to $\aleph_1$ for every stage $\alpha$ with $\alpha+\omega_1\leq\lambda$, which renders this method useless in the construction of models of forcing axioms with large continuum.}

 We remind the reader that a forcing $\bbP$ is said to be \emph{$\aleph_2$-Knaster} in case for every sequence $(p_i\,:\, i<\omega_2)$ of conditions in $\bbP$ there is $I\sub \omega_2$ of size $\aleph_2$ such that $p_{i_0}$ and $p_{i_1}$ are compatible in $\bbP$ for all $i_0$, $i_1\in I$. Obviously, every $\aleph_2$-Knaster forcing has the $\aleph_2$-chain condition.

In this paper we prove that $\PFA(\omega_1)$ is consistent with  the continuum being arbitrary large, thereby answering the above question. Given a class $\Gamma$ of forcing notions and a cardinal $\lambda$, the forcing axiom $\FA(\Gamma)_\lambda$ is the statement that for every $\bbP\in \Gamma$ and every collection $\{D_i\,:\, i<\lambda\}$ of dense subsets of $\bbP$ there is a filter $G\sub\bbP$ such that $G\cap D_i\neq\emptyset$ for all $i<\lambda$. In fact, we will prove the consistency of $2^{\aleph_0}$ being an arbitrarily fixed regular cardinal $\kappa$ together with $\PFA(\omega_1)_{{<}\kappa}$, where $\PFA(\omega_1)_{{<}\kappa}$ is $\FA(\{\bbP\,:\, \bbP\text{ proper}, |\bbP|=\aleph_1\})_\lambda$ for all $\lambda<\kappa$.

Our first main theorem is the following.

\begin{theorem}
\label{main}
Assume $\GCH$. Let $\kappa\geq\aleph_2$ be a regular cardinal.
Then there is an $\aleph_2$-Knaster proper partial order $\bbP$ forcing the following statements.
\begin{enumerate}
\item $2^{\aleph_0}=2^{\aleph_1}=\kappa$
\item $\PFA(\omega_1)_{{<}\kappa}$
\end{enumerate}
\end{theorem}

The forcing witnessing Theorem \ref{main} is a finite-support iteration with systems of models with markers as side conditions, in the style of the constructions in \cite{aspero-mota1}, \cite{aspero-mota2} or \cite{fnr}. However, we were also inspired by Shelah's memory iteration technique (s.\ for example  \cite{sh592}, \cite{sh619}, \cite{sh684}, \cite{sh1102} and \cite{gilton2}).

Familiarity with proper forcing should be enough to follow the paper. Some familiarity with the method of forcing with symmetric systems, as presented for example in \cite{aspero-mota1}, and with some of the arguments from \cite{fnr}, might also be useful. Our notation is standard (see for example \cite{jech}); in particular, given a forcing notion $\bbP$ and two forcing condition $p$, $q \in \bbP$, we use $q\leq_{\bbP} p$
to mean that $q$ is stronger than $p$. Also, given forcing notions $\bbP$ and $\bbQ$, we write $\bbP\lessdot \bbQ$ to denote that $\bbP$ is a complete suborder of $\bbQ$ (i.e., $\bbP$ is a suborder of $\bbQ$, any two incompatible conditions in $\bbP$ are incompatible in $\bbQ$, and any maximal antichain in $\bbP$ is in fact also a maximal antichain in $\bbQ$).

As it turns out, a small variant of our construction gives rise to a model satisfying a useful generic absoluteness statement, in addition to the statements in the conclusion of Theorem \ref{main} (this is pursued in Section \ref{more}). This form of generic absoluteness implies, among other principles,
 Todor\v{c}evi\'{c}'s Open Colouring Axiom for sets of size $\aleph_1$,  the Abraham-Rubin-Shelah Open Colouring Axiom, Todor\v{c}evi\'{c}'s  P-ideal Dichotomy for $\aleph_1$-generated ideals on $\omega_1$, Moore's Measuring principle, and Baumgartner's Thinning-Out Principle. Hence, all these principles are simultaneously compatible with $2^{\aleph_0}>\aleph_2$.

 Finally, we will show that a further, very mild, modification of our second construction gives rise to a model satisfying $\MM(\omega_1)$, i.e., Martin's Maximum restricted to forcings of size $\aleph_1$, in addition to all conclusions in the second theorem. This result shows the consistency of $\MM(\omega_1)$ together will $2^{\aleph_0}$ being arbitrarily large, thus answering a question from  \cite{DKMMZ} and \cite{foreman}.

Theorem \ref{main}---as well as the results from \cite{aspero-mota1}, \cite{aspero-mota2}, etc., and all known derivations of $2^{\aleph_0}=\aleph_2$ from forcing axioms---strongly suggest that the restriction of $\PFA$ to a class $\mathcal K$ of proper posets should decide the size of the continuum to be $\aleph_2$ only if $\mathcal K$ contains enough forcing notions collapsing cardinals to $\aleph_1$. This motivates the following general question.

\begin{question} Is the Proper Forcing Axiom restricted to the class of cardinal-preserving posets compatible with $2^{\aleph_0}>\aleph_2$?  Is even the Proper Forcing Axiom restricted to the class of posets with the $\aleph_2$-chain condition compatible with $2^{\aleph_0}>\aleph_2$?  \end{question}

\begin{remark}
In \cite{aspero-tananimit} it is proved that the forcing axiom $$\FA(\{\{\bbP\,:\, \bbP\text{ preserves stationary subsets of $\omega_1$ and has the $\aleph_2$-c.c.}\})_{\aleph_2}$$ is inconsistent.
\end{remark}

The rest of the paper is structured as follows. In Section \ref{mwithnch} we prove Theorem \ref{main}. We start Section \ref{more} by considering a slight variation of the construction from Theorem \ref{main} and show that it satisfies the form of generic absoluteness we have referred to above (in Subsection \ref{provingthm2}). This is Theorem \ref{thm2}. Then, in Subsection \ref{applications}, we prove that a number of classical consequences of $\PFA$ follow from our form of generic absoluteness. Finally, in Section \ref{section3} we prove that a small modification to our construction for Theorem \ref{thm2} gives rise to a forcing extension satisfying also $\MM(\omega_1)$.

\textbf{Acknowledgements}: We thank Tadatoshi Miyamoto, Miguel Angel Mota, and a referee for helpful suggestions and for pointing out errors in earlier versions of this paper.

\section{A model of $\PFA(\omega_1)_{{<}2^{\aleph_0}}$ and $2^{\aleph_0}$ arbitrarily large}
\label{mwithnch}

In this section we prove Theorem \ref{main}.

Assume $\GCH$ holds and let $\kappa\geq \aleph_2$ be a regular cardinal. For every cardinal $\theta$ let $\cH(\theta)$ denote the set of all sets that are hereditarily of cardinality less than $\theta$. Let $\phi: \kappa \to \cH(\kappa)$
be a function such that the set $\phi^{-1}(x) \subseteq \kappa$ is unbounded for every $x \in \cH(\kappa)$.\footnote{$\phi$ exists since $|\cH(\kappa)|=\kappa$ by $\GCH$.} The function $\phi$ will be our book-keeping function.
Let also $\triangleleft$ be a well-order of $\cH(\kappa^+)$ in order type $\kappa^+$ and let us note that $\triangleleft\in\cH(\kappa^{++})$. Given an ordinal $\alpha$, we will follow the standard practice of denoting the $\alpha$-th cardinal past $\kappa$ by $\kappa^{+\alpha}$; in other words, $\kappa^{+0}=\kappa$, $\kappa^{+(\beta+1)}=(\kappa^{+\beta})^+$, and $\kappa^{+\alpha}=\sup_{\beta<\alpha}\kappa^{+\beta}$ if $\alpha$ is a limit ordinal.

Given a set $N$, $\delta_N$ is defined as $N\cap\omega_1$. When $\delta_N\in\omega_1$, this ordinal is sometimes called \emph{the height of $N$}. If $N_0$ and $N_1$ are $\in$-isomorphic models of the Axiom of Extensionality, we refer to the unique isomorphism $\Psi:(N_0; \in)\rightarrow (N_1; \in)$ by $\Psi_{N_0, N_1}$.

The following notion is defined in \cite{aspero-mota1}.

\begin{definition}
Given a predicate $\Phi\sub\cH(\kappa)$, a finite set $\mtcl N\sub [\cH(\kappa)]^{\aleph_0}$ is a \emph{$\Phi$-symmetric system} if the following holds.

\begin{enumerate}
\item For every $N\in\mtcl N$, $(N; \in, \Phi\cap N)\prec (\cH(\kappa); \in, \Phi)$.
\item For all $N_0$, $N_1\in\mtcl N$, if $\delta_{N_0}=\delta_{N_1}$, then $(N_0; \in, \Phi\cap N_0)\cong (N_1; \in, \Phi\cap N_1)$. Moreover, $\Psi_{N_0, N_1}$ is the identity on $N_0\cap N_1$.
\item For all $N_0$, $N_1\in\mtcl N$ and all $M\in\mtcl N\cap N_0$, if $\delta_{N_0}=\delta_{N_1}$, then $\Psi_{N_0, N_1}(M)\in\mtcl N$.
\item For all $N_0$, $N_1\in\mtcl N$, if $\delta_{N_0}<\delta_{N_1}$, then there is some $N_1'\in\mtcl N$ such that $\delta_{N_1'}=\delta_{N_1}$ and $N_0\in N_1'$.
\end{enumerate}
\end{definition}

The following two amalgamation lemmaa are proved in \cite{aspero-mota1}.

\begin{lemma}\label{amalg}
Let $\Phi\sub\cH(\kappa)$, $\mtcl N$ a $\Phi$-symmetric system, $N\in\mtcl N$, and $\mtcl M\in N$ a $\Phi$-symmetric system such that $\mtcl N\cap N\sub \mtcl M$. Let $$\mtcl W=\mtcl N\cup\mtcl M\cup\{\Psi_{N, N'}(M)\,:\, N'\in\mtcl N,\,\delta_{N'}=\delta_N,\,M\in\mtcl M\}$$ Then $\mtcl W$ is a $\Phi$-symmetric system.
\end{lemma}

\begin{lemma}\label{amalg2}
Let $\Phi\sub\cH(\kappa)$, let $\mtcl N_0$ and $\mtcl N_1$ be two $\Phi$-symmetric systems, and suppose there are, for some $n<\omega$, enumerations $(N^0_i)_{i<n}$ and $(N^1_i)_{i<n}$ of $\mtcl N_0$ and $\mtcl N_1$, respectively, such that the structures $$\langle \bigcup\mtcl N_0; \in,\Phi\cap (\bigcup\mtcl N_0), N^0_i\rangle_{i<n}$$ and $$\langle \bigcup\mtcl N_1; \in,\Phi\cap (\bigcup\mtcl N_1), N^1_i\rangle_{i<n}$$ are isomorphic, with the isomorphism between them being the identity on the intersection $(\bigcup\mtcl N_0)\cap (\bigcup\mtcl N_1)$. Then  $\mtcl N_0\cup\mtcl N_1$ is a $\Phi$-symmetric system.
\end{lemma}

For every $\rho<\kappa$, let $\mtcl E_\rho$ be the $\triangleleft$-first club of $[\cH(\kappa)]^{\aleph_0}$ with the property that for every $N\in \mtcl E_\rho$ there is some $N^*\prec \cH(\kappa^{+(1+\rho+1)})$ such that $\phi$, $\triangleleft\in N^*$ and $N^*\cap \cH(\kappa)=N$.

A \emph{model with marker} is a pair $(N, \rho)$, where
\begin{itemize}
\item $N\in [\cH(\kappa)]^{\aleph_0}$,
\item $\rho\in N\cap\kappa$, and
\item  $N\in\mtcl E_\rho$.
\end{itemize}

\begin{remark}
\begin{enumerate}
\item If $(N, \rho)$ is a model with marker and $\bar\rho\in N\cap\rho$, then $(N, \bar\rho)$ is also a model with marker. To see this, we note that if $N^*\prec\cH(\kappa^{+(1+\rho+1)})$ is such that  $\phi$, $\triangleleft\in N^*$ and $N=N^*\cap\cH(\kappa)$, then $\mtcl E_{\bar\rho}\in N^*$ since this club is definable over $\cH(\kappa^{+(1+\rho+1)})$ from the parameters $\phi$, $\triangleleft$ and $\cH(\kappa^{+(1+\bar\rho+1)})$, all of which are in $N^*$, and therefore $N\in \mtcl E_{\bar\rho}$.
\item Alternatively, we could have dispensed with $\triangleleft$ and the clubs $\mathcal E_\rho$; instead, we could have worked with just a sequence $\langle\Phi_\alpha\,:\, \alpha<\kappa\rangle$ of increasingly expressive predicates  contained in $\cH(\kappa)$ similar to the one we are about to define---as in, for example, \cite{fnr}. On the other hand, it is in any case convenient to have $\triangleleft$ available in Section \ref{more}.
\end{enumerate}
\end{remark}

Our forcing witnessing Theorem \ref{main} will be $\bbP_\kappa$ for a certain sequence $\langle \bbP_\alpha\,:\,\alpha\leq\kappa\rangle$ of forcing notions we will soon define by recursion on $\alpha$.
Together with $\langle \bbP_\alpha\,:\,\alpha\leq\kappa\rangle$, we will define a sequence $\langle\Phi_\alpha\,:\, \alpha<\kappa\rangle$ of predicates of $\cH(\kappa)$. To start with, $\Phi_0$ is the satisfaction predicate for $(\cH(\kappa); \in, \phi)$.
Given $\alpha<\kappa$, and assuming $\bbP_\alpha$ has been defined, we let $\Phi_{\alpha+1}$ be a predicate of $\cH(\kappa)$ encoding, in some fixed canonical way, the satisfaction predicate of $$(\cH(\kappa); \in, \{(\beta, x, N)\,:\, \beta\leq\alpha, x\in\Phi_\beta, N\in \mtcl E_\beta\}, \bbP_\alpha, \Vdash_\alpha^\ast).$$ Here, $\Vdash_\alpha^*$ is the forcing relation for $\bbP_\alpha$ for formulas involving names in $\cH(\kappa)$. Also, if $\alpha<\kappa$ is a limit ordinal and $\Phi_\beta$ has been defined for all $\beta<\alpha$, $$\Phi_\alpha=\{(\beta, x)\,:\, \beta<\alpha, x\in\Phi_\beta\}.$$

We define a \emph{symmetric system of models with makers} to be a collection $\Delta$ of models with markers such that:

\begin{itemize}
\item[(A)] $\dom(\Delta)=\{N\,:\, (N, \rho)\in \Delta\text{ for some }\rho\}$ is a $\Phi_0$-symmetric system;
\item[(B)] for every $\rho<\kappa$, $\{N\,:\,(N, \rho)\in\Delta\}$ is a $\Phi_\rho$-symmetric system.
\end{itemize}

Given $\alpha<\kappa$, $\bbP_\alpha$ will consist of pairs $p=(F_p, \Delta_p)$, where:
\begin{enumerate}
\item $F_p$ is a finite function with $\dom(F_p)\sub\alpha$ and such that for each $\beta\in \dom(F_p)$, $F_p(\beta)\in \omega_1$.
\item $\Delta_p$ is a symmetric system of models with markers $(N, \rho)$ such that $\rho\leq\alpha$.
\end{enumerate}

\begin{remark}
Modulo some notational changes, our construction will in fact be, essentially, a small variation of (a simple version of) the main forcing construction from \cite{aspero-mota1}.  One important change with respect to that construction is that we now drop the requirement that if $(N, \rho)\in\Delta_p$ and $\bar\rho\in N\cap \rho$, then also $(N, \bar\rho)\in \Delta_p$. Another essential change  is that, given a model with marker $(N, \beta+1)\in \Delta_p$, we require that the working part $F_p(\beta)$ be forced to be, not $(N[\name{G}_{\bbP_\beta}], \name{\bbQ}_\beta)$-generic, but $(N[\name{G}_{\bbP_\beta\restriction\,\mtcl U^\beta}], \name{\bbQ}_\beta)$-generic for a certain appropriate complete suborder $\bbP_\beta\restriction\,\mtcl U^\beta$ of $\bbP_\beta$. We will then of course have that $\name{\bbQ}_\beta$ is in fact a $\bbP_\beta\restriction\,\mtcl U^\beta$-name.
\end{remark}

The above specification defines the universe of $\bbP_0$.

For all $\alpha<\kappa$, $p\in\bbP_\alpha$, and $\beta<\alpha$, we define $p\restriction \beta = (F_p\restriction \beta, \Delta_p\restriction\beta)$, where $$\Delta_p\restriction\beta =\{(N, \rho)\in\Delta_p\,:\,\rho\leq\beta\}$$

Given $\alpha<\kappa$, the extension relation on $\bbP_\alpha$ will be denoted by $\leq_\alpha$. $p_1\leq_\alpha p_0$ will mean that $p_1$ is stronger than $p_0$ in $\bbP_\alpha$.

Given $p_0$, $p_1\in\bbP_0$, $p_1\leq_0 p_0$ iff $\Delta_{p_0}\sub \Delta_{p_1}$.

Given any $\alpha<\kappa$, we associate to $\alpha$ sets $\overline{\mtcl U}^\alpha$, $\mtcl U^\alpha\in [\alpha]^{{\leq}\aleph_1}$
 defined by letting
$\overline{\mtcl U}^\alpha=\mtcl U^\alpha=\emptyset$ if $\phi(\alpha)$ is not a sequence $\langle A^\alpha_i\,:\, i<\omega_1\rangle$ of antichains of $\bbP_\alpha$, each of size $\leq \aleph_1$, and, in the other case, if $\phi(\alpha)=\langle A^\alpha_i\,:\, i<\omega_1\rangle$, letting $\overline{\mtcl  U}^\alpha$ be the union of

\begin{itemize}
\item $\bigcup\{\dom(F_p)\,:\,p\in\bigcup_{i<\omega_1}A^\alpha_i\}$ and
\item $\bigcup\{N\cap\rho\,:\,p\in\bigcup_{i<\omega_1}A^\alpha_i,\,(N, \rho)\in\Delta_p\}$\footnote{There is room for flexibility in these definitions. For example we could have replaced this second bullet point with $\bigcup\{N\cap\alpha\,:\,p\in\bigcup_{ i<\omega_1}A^\alpha_i,\,(N, \rho)\in\Delta_p\}$, or $\{\rho\,:\, p\in\bigcup_{ i<\omega_1}A^\alpha_i,\,(N, \rho+1)\in\Delta_p\}$, and everything would have worked the same.}
\end{itemize}
\noindent and letting $\mtcl U^\alpha=\overline{\mtcl U}^\alpha\cup\bigcup\{\mtcl U^\xi\,:\,\xi\in\overline{\mtcl U}^\alpha\}$.

\begin{remark}\label{rmk00}
The following is easily proved by induction on $\alpha<\kappa$.
\begin{enumerate}
\item $|\mtcl U^\alpha|\leq\aleph_1$;
\item for each $\beta\in\mtcl U^\alpha$, $\mtcl U^\beta\sub\mtcl U^\alpha$.
\end{enumerate}
\end{remark}

Given any $\alpha<\kappa$, and assuming $\bbP_\alpha$ has been defined, we define $\bbP_\alpha\restriction\,\mtcl U^\alpha$ to be the suborder of $\bbP_\alpha$ consisting of those $p\in\bbP_\alpha$ such that

\begin{itemize}
\item $\dom(F_p)\sub\mtcl U^\alpha$ and
\item $\rho\in\mtcl U^\alpha$ for all $(N, \rho)\in \Delta_p$.
\end{itemize}

Given a $\bbP_\alpha$-condition $p$, $p\restriction \mtcl U^\alpha$ is defined as $$(F_p\restriction\mtcl U^\alpha, \{(N, \rho)\,:\, (N, \rho)\in\Delta_p,\,\rho\in\mtcl U^\alpha\}).$$
We note that $p\restriction\mtcl U^\alpha$ is a condition in $\bbP_\alpha\restriction\mtcl U^\alpha$.

Suppose $\name{\bbQ}$ is a $\bbP_\alpha\restriction\,\mtcl U^\alpha$-name for a partial order on $\omega_1^{\bold V}$. Given a $\bbP_\alpha\restriction\,\mtcl U^\alpha$-generic filter $G_0$ over $\bold V$, we say that \emph{$\name{\bbQ}_{G_0}$ is $(G_0, \bbP_\alpha)$-proper in $\bold V[G_0]$} in case there is a club $E$ of $[\cH(\kappa)]^{\aleph_0}$ in $\bold V$ with the property that for all $N\in E$, if there is some $q\in \bbP_\alpha$ such that $q\restriction\mtcl U^\alpha\in G_0$ and $(N, \rho)\in\Delta_q$
for all $\rho\in N\cap \mtcl U^\alpha$, then for every $\nu\in \omega_1^{\bold V}\cap N[G_0]$ there is some $(N[G_0], \name{\bbQ}_{G_0})$-generic condition $\nu^*\in\omega_1^{\bold V}$ such that $\nu^*\leq_{\name{\bbQ}_{G_0}} \nu$. And of course, we say that $\name{\bbQ}$ is \emph{forced to be $(\name{G}_{\bbP_\alpha\restriction\,\mtcl U^\alpha}, \bbP_\alpha)$-proper} if $\bbP_\alpha\restriction\,\mtcl U^\alpha$ forces over $\bold V$ that $\name\bbQ$ is $(\name{G}_{\bbP_\alpha\restriction\,\mtcl U^\alpha}, \bbP_\alpha)$-proper in $\bold V[\name{G}_{\bbP_\alpha\restriction\,\mtcl U^\alpha}]$.

We also define $\name{\bbQ}_\alpha$ to be the following $\bbP_\alpha\restriction\,\mtcl U^\alpha$-name:

\begin{itemize}
\item if $\phi(\alpha)$ is a sequence $\langle A^\alpha_i\,:\, i<\omega_1\rangle$ of antichains of $\bbP_\alpha$, each of size $\leq\aleph_1$,  such that $\bigcup_{i<\omega_1}\{i\}\times A^\alpha_i$, viewed as a nice $\bbP_\alpha\restriction\,\mtcl U^\alpha$-name for a subset of $\omega_1$, canonically encodes a forcing notion on $\omega_1^{\bold V}$ which $\bbP_\alpha\restriction\,\mtcl U^\alpha$ forces to be $(\name{G}_{\bbP_\alpha\restriction\,\mtcl U^\alpha}, \bbP_\alpha)$-proper, then $\name{\bbQ}_\alpha$ is a $\bbP_\alpha\restriction\,\mtcl U^\alpha$-name for this forcing notion;
\item in the other case, $\name{\bbQ}_\alpha$ is a $\bbP_\alpha\restriction\,\mtcl U^\alpha$-name for trivial forcing $\{\emptyset\}$.
\end{itemize}

We are now in a position to define $\bbP_\alpha$ for $\alpha>0$. A condition in $\bbP_\alpha$ is a pair $p=(F_p, \Delta_p)$ satisfying (1) and (2) above, together with the following.

\begin{enumerate}

\item[(3)] For each $\beta<\alpha$ and $(N, \beta+1)\in\Delta_p$, $\mtcl U_\beta\cap N\sub\{\rho\,:\,(N, \rho)\in\Delta_p\}$.

\item[(4)] For each $\beta\in\dom(F_p)$ and $(N, \beta+1)\in\Delta_p$, $p\restriction \mtcl U^\beta$ forces in $\bbP_\beta\restriction\,\mtcl U^\beta$ that $F_p(\beta)$ is $(N[\name{G}_{\bbP_\beta\restriction\,\mtcl U^\beta}], \name{\bbQ}_\beta)$-generic.

\end{enumerate}

Given $p_0$ and $p_1$, $\bbP_\alpha$-conditions, $p_1\leq_\alpha p_0$ iff

\begin{itemize}
\item for each $\beta<\alpha$, $p_1\restriction\beta\leq_\beta p_0\restriction\beta$;
\item $\{N\,:\,(N, \alpha)\in\Delta_{p_0}\}\sub\{N\,:\, (N, \alpha)\in\Delta_{p_1}\}$;
\item if $\alpha=\alpha_0+1$ and $\alpha_0\in\dom(F_{p_0})$, then
\begin{enumerate}
\item[(i)] $\alpha_0\in\dom(F_{p_1})$ and
\item[(ii)] $p_1\restriction \alpha_0\Vdash_{\bbP_{\alpha_0}}F_{p_1}(\alpha_0)\leq_{\name{\bbQ}_{\alpha_0}}F_{p_0}(\alpha_0)$.
\end{enumerate}\end{itemize}

Finally, we define $\bbP_\kappa=\bigcup_{\alpha<\kappa}\bbP_\alpha$.

It is immediate to see that $\langle \bbP_\alpha\,:\,\alpha\leq\kappa\rangle$ is a forcing iteration, in the sense that $\bbP_\alpha\lessdot\bbP_\beta$ holds for all $\alpha<\beta$. In fact, we have the following, which is easily verified 
going one by one through all clauses in the definition of condition for the pair $(F_q\cup (F_p\restriction [\beta,\,\alpha)), \Delta_p\cup\Delta_q)$.

\begin{lemma}\label{compl0}
For all $\beta<\alpha<\kappa$, $p\in\bbP_\alpha$, and $q\in \bbP_\beta$, if $q\leq_\beta p\restriction \beta$, then $$(F_q\cup (F_p\restriction [\beta,\,\alpha)), \Delta_p\cup\Delta_q)$$ is a condition in $\bbP_\alpha$ extending both $p$ and $q$.
\end{lemma}

We also have, for any given $\alpha<\kappa$, that $\bbP_\alpha\restriction\,\mtcl U^\alpha$ is a complete suborder of $\bbP_\alpha$. In fact, as the following  lemma shows (which easily follows from Remark \ref{rmk00}), the function sending $p\in\bbP_\alpha$ to $p\restriction\mtcl U^\alpha$ is a projection from $\bbP_\alpha$ onto $\bbP_\alpha\restriction\,\mtcl U^\alpha$.

\begin{lemma}\label{compl1}
Let $\alpha<\kappa$, $p\in\mtcl P_\alpha$, and let $q\in\bbP_\alpha\restriction\,\mtcl U^\alpha$ be a condition extending $p\restriction\,\mtcl U^\alpha$. Let $$r=((F_p\restriction \alpha\setminus\mtcl U^\alpha) \cup F_q, \Delta_p\cup\Delta_q).$$

Then $r$ is a condition in $\bbP_\alpha$ extending both $p$ and $q$.
\end{lemma}

\begin{proof}
We prove by induction on $\beta\leq\alpha$ that $r\restriction\beta$ is a condition in $\bbP_\beta$. This is enough, as it is then immediate that $r$ extends both $p$ and $q$. Clauses (1) and (2) in the definition of $\bbP_\beta$ for $r\restriction\beta$ are trivial. Clause (3) follows from the fact that $\mathcal U^{\gamma_1}\sub\mathcal U^{\gamma_0}$ for all $\gamma_0<\beta$ and all $\gamma_1\in\mtcl U^{\gamma_0}$
and clause (4) is true by the induction hypothesis and the fact that $(r\restriction\beta)\restriction\mathcal U^\gamma=r\restriction\mathcal U^\gamma$ for every $\gamma<\beta$.
\end{proof}

In what follows we will use Lemmas \ref{compl0} and \ref{compl1} repeatedly, often without mention.

We will next show, by a standard $\Delta$-system argument using $\CH$, that each $\bbP_\alpha$ is $\aleph_2$-Knaster.

\begin{lemma}\label{cc} For each $\alpha\leq\kappa$, $\bbP_\alpha$ is $\aleph_2$-Knaster.
\end{lemma}

\begin{proof}
Let $p_i\in\bbP_\alpha$ for each $i<\omega_2$. For each $i$ let $M_i$ be a countable elementary submodel of $\cH(\kappa^+)$ containing $\phi$ and $p_i$, and let $\mtcl M_i$ be a structure with universe $M_i$ coding $p_i$ and $\{(\beta, x)\,:\, \beta\in M_i\cap(\alpha+1), x\in \Phi_\beta\cap M_i\}$ in some canonical way. By $\CH$, we may find $I\in [\omega_2]^{\aleph_2}$ and $R\in [\cH(\kappa)]^{\aleph_0}$  such that $\{M_i\,:\, i\in I\}$ forms a $\Delta$-system with root $R$. By further shrinking $I$ if necessary, we may assume that for all $i_0$, $i_1\in I$, $\mtcl M_{i_0}\cong \mtcl M_{i_1}$ and $\Psi_{M_{i_0}, M_{i_1}}$ is the identity on $M_{i_0}\cap M_{i_1}$.

It is then straightforward to see, using Lemma \ref{amalg2}, that if $i_0$, $i_1\in I$, then the ordered pair $(F_{p_{i_0}}\cup F_{p_{i_1}}, \Delta_{p_{i_0}}\cup\Delta_{p_{i_1}})$ is a condition in $\bbP_\alpha$ extending both $p_{i_0}$ and $p_{i_1}$.
\end{proof}

The following symmetry lemma, whose proof is straightforward, will be quite useful.

\begin{lemma}\label{symm-lem}
Let $\alpha<\kappa$, $p\in\bbP_\alpha$, and $M^0$ and $M^1$ in $\dom(\Delta_p)$ such that $\alpha\in M^0\cap M^1$ and $\delta_{M^0}=\delta_{M^1}$. Suppose $(M^0, \rho)$ and $(M^1, \rho)$ are in $\Delta_p$ for every $\rho\in\mtcl U^\alpha\cap M^0$. Suppose $s\in M^0\cap (\bbP_\alpha\restriction\,\mtcl U^\alpha)$ is such that $p\restriction\mtcl U^\alpha\leq_{\bbP_\alpha\restriction\,\mtcl U^\alpha}s$. Then $p\restriction\mtcl U^\alpha\leq_{\bbP_\alpha\restriction\,\mtcl U^\alpha}\Psi_{M^0, M^1}(s)$.
\end{lemma}

\begin{proof} The point is that for every $(M, \rho)\in\Delta_s$, $\rho\in\mtcl U^\alpha\cap M^0\subseteq M^1$ (since $\mtcl U^\alpha\in M^0\cap M^1$ has size at most $\aleph_1$ and $\delta_{M^0}=\delta_{M^1}$) and therefore $(\Psi_{M^0, M^1}(M), \rho)\in\Delta_p$ by the closure condition in the definition of symmetric system of models with markers.
\end{proof}

We also have the following lemma.

\begin{lemma}\label{symm-small-proper}
Suppose $\bbQ$ is a proper forcing on $\omega_1$ and $\mtcl N$ is a finite symmetric system of countable elementary submodels of $\cH(\theta)$, for some cardinal $\theta\geq\omega_2$, containing $\bbQ$. If $\nu\in\delta_{N_0}$ for some $N_0\in\mtcl N$ of minimal height within $\mtcl N$, then there is an extension of $\nu$ in $\bbQ$ which is $(N, \bbQ)$-generic for every $N\in\mtcl N$.
\end{lemma}

\begin{proof}
Let $(N_k)_{k\leq n}$ be, for some $n<\omega$, an $\in$-chain of members of $\mtcl N$ such that every member of $\mtcl N$ is isomorphic to $N_k$ for some $k$. Let also $N_{n+1}=\cH(\theta)$. Noting then that $\nu\in N_0$, we build a decreasing sequence $(\nu_k)_{k\leq n+1}$ of conditions in $\bbQ$ by setting $\nu_0=\nu$ and letting $\nu_{k+1}$ be, for each $k\leq n$, some $(N_k, \bbQ)$-generic condition in $N_{k+1}$ extending $\nu_k$.

We claim that $\nu^*=\nu_{n+1}$ satisfies the conclusion. For this, let $N\in\mtcl N$, $D\in N$ a dense subset of $\bbQ$, $\nu'$ an extension of $\nu^*$ in $\bbQ$, and let us find some condition in $\bbQ$ extending both $\nu'$ and some condition in $D\cap\delta_N$. Let $k\leq n$ be such that $N\cong N_k$ and let $E=\Psi_{N, N_k}(D)$. Then $E$ is a dense subset of $\bbQ$ and hence there is a common extension $\nu''$ of $\nu'$ and some $\bar\nu\in E\cap \delta_{N_k}=E\cap\delta_N$. But of course $\bar\nu=\Psi_{N_k, N}(\bar\nu)\in D$, which finishes the proof.
\end{proof}

It will be crucial to have the following lemma at hand.

\begin{lemma}\label{same-subsets-o1-ext}
Let $\alpha<\kappa$. Let $t\in\bbP_\alpha$ and $(M^0, \alpha)$, $(M^1, \alpha)\in\Delta_t$ with the following properties.
\begin{enumerate}
\item $(M^0, \rho_0)\in\Delta_t$ and $(M^1, \rho_1)\in \Delta_t$ for all $\rho_0\in M^0\cap \mtcl U^\alpha$ and $\rho_1\in M^1\cap \mtcl U^\alpha$.
\item $(M^0; \in, \Phi_{\alpha+1}\cap M^0)$ and $(M^1; \in, \Phi_{\alpha+1}\cap M^1)$ are both elementary submodels of $(\cH(\kappa); \in, \Phi_{\alpha+1})$.
\item $(M^0; \in, \Phi_{\alpha+1}\cap M^0)\cong (M^1; \in, \Phi_{\alpha+1}\cap M^1)$.\footnote{In particular,  $\delta_{M^0}=\delta_{M^1}$.}
\item  $t\restriction\mtcl U^{\alpha}$ is both $(M^0, \bbP_\alpha\restriction\,\mtcl U^\alpha)$-generic and  $(M^1, \bbP_\alpha\restriction\,\mtcl U^\alpha)$-generic.
\end{enumerate}

Then $t\restriction\mtcl U^\alpha$ forces in $\bbP_\alpha\restriction\,\mtcl U^\alpha$ that $(M^0[\name{G}_{\bbP_\alpha\restriction\,\mtcl U^\alpha}]; \in)\cong (M^1[\name{G}_{\bbP_\alpha\restriction\,\mtcl U^\alpha}]; \in)$ as witnessed by the function sending $(\name X)_{\name{G}_{\bbP_\alpha\restriction\,\mtcl U^\alpha}}$, for a $\bbP_\alpha\restriction\,\mtcl U^\alpha$-name $\name X\in M^0$, to $(\Psi_{M^0, M^1}(\name X))_{\name{G}_{\bbP_\alpha\restriction\,\mtcl U^\alpha}}$.
\end{lemma}

\begin{proof}
Suppose $\name X$, $\name Y\in M^0$ are $\bbP_\alpha\restriction\,\mtcl U^\alpha$-names and some condition $t'\in\bbP_\alpha\restriction\,\mtcl U^\alpha$ such that $t'\leq_{\bbP_\alpha\restriction\,\mtcl U^\alpha} t\restriction\mathcal U^\alpha$ forces $\name X\in\name Y$. By the $(M^0, \bbP_\alpha\restriction\,\mtcl U^\alpha)$-genericity of $t\restriction\mathcal U^\alpha$ we may assume, after extending $t'$ if necessary, that $t'$ extends some $s\in \bbP_\alpha\restriction\,\mtcl U^\alpha\cap M^0$ such that $s\Vdash_{\bbP_\alpha\restriction\,\mtcl U^\alpha}\name X\in\name Y$.

Let $\Vdash^*$ denote the forcing relation for $\bbP_\alpha\restriction\,\mtcl U^\alpha$ restricted to formulas with names in $\cH(\kappa)$. Since $\Vdash^*$ is definable from $\Vdash^*_\alpha$ and $\bbP_\alpha$ and $\Psi_{M^0, M^1}$ is an isomorphism between $(M^0; \in, \Phi_{\alpha+1}\cap M^0)$ and $(M^1; \in, \Phi_{\alpha+1}\cap M^1)$, it follows (since of course $\name X$ and $\name Y$ are both in $\cH(\kappa)$) that $$\Psi_{M^0, M^1}(s)\Vdash_{\bbP_\alpha\restriction\,\mtcl U^\alpha}\Psi_{M^0, M^1}(\name X)\in \Psi_{M^0, M^1}(\name Y)$$
Since $t'\leq_{\bbP_\alpha\restriction\,\mtcl U^\alpha} \Psi_{M^0, M^1}(s)$ by Lemma \ref{symm-lem}, it follows that $t'\Vdash_{\bbP_\alpha\restriction \,\mtcl U^\alpha}\Psi_{M^0, M^1}(\name X)\in \Psi_{M^0, M^1}(\name Y)$. And arguing similarly, if $t'\Vdash_{\bbP_\alpha\restriction \,\mtcl U^\alpha}\name X\notin \name Y$, then also $t'\Vdash_{\bbP_\alpha\restriction \,\mtcl U^\alpha}\Psi_{M^0, M^1}(\name X)\notin \Psi_{M^0, M^1}(\name Y)$, and if $t'\Vdash_{\bbP_\alpha\restriction \,\mtcl U^\alpha}\name X= \name Y$, then also $t'\Vdash_{\bbP_\alpha\restriction \,\mtcl U^\alpha}\Psi_{M^0, M^1}(\name X)= \Psi_{M^0, M^1}(\name Y)$.

The above argument shows that the function sending $(\name X)_{\name{G}_{\bbP_\alpha\restriction\,\mtcl U^\alpha}}$, for a $\bbP_\alpha\restriction\,\mtcl U^\alpha$-name $\name X\in M^0$, to $(\Psi_{M^0, M^1}(\name X))_{\name{G}_{\bbP_\alpha\restriction\,\mtcl U^\alpha}}$, is forced by $t\restriction\mathcal U^\alpha$ to be a well-defined isomorphism between  $(M^0[\name{G}_{\bbP_\alpha\restriction\,\mtcl U^\alpha}]; \in)$ and $(M^1[\name{G}_{\bbP_\alpha\restriction\,\mtcl U^\alpha}]; \in)$.
\end{proof}

The following is our properness lemma.

\begin{lemma}\label{properness}
Let $\alpha<\kappa$. Then the following holds.

\begin{enumerate}
\item[$(1)_\alpha$] Let $p\in\bbP_\alpha$, let $N^*$ be a countable elementary submodel of $\cH(\kappa^{+(1+\alpha+1)})$ such that $\phi$, $\triangleleft\in N^*$, let $N=N^*\cap\cH(\kappa)$, and suppose $(N, \rho)\in\Delta_p$ for all $\rho\in N\cap (\alpha+1)$. Then $p$ is $(N^*, \bbP_\alpha)$-generic.

\item[$(2)_\alpha$] Let $p\in\bbP_\alpha\restriction\,\mtcl U^\alpha$, let $N^*$ be a countable elementary submodel of $\cH(\kappa^{+(1+\alpha+1)})$ such that $\phi$, $\triangleleft\in N^*$, let $N=N^*\cap\cH(\kappa)$, and suppose $(N, \rho)\in\Delta_p$ for all $\rho\in N\cap \mtcl U^\alpha$. Then $p$ is $(N^*, \bbP_\alpha\restriction\,\mtcl U^\alpha)$-generic.
\end{enumerate}
\end{lemma}

\begin{proof}
The proof will be by induction on $\alpha$. We will only give the proof of $(1)_\alpha$ explicitly as the proof of $(2)_\alpha$ is essentially the same.

For $\alpha=0$, the conclusion  of $(1)_\alpha$ follows trivially from Lemma \ref{amalg}.
 We may thus assume in what follows that $\alpha>0$.

Suppose first that $\alpha$ is a successor ordinal, $\alpha=\alpha_0+1$. Let $D\in N^*$ be an open dense subset of $\bbP_\alpha$ and let $p'\in\bbP_\alpha$ be a condition extending $p$. By further extending $p'$ if necessary, we may assume that $p'\in D$. We may assume that $\alpha_0\in \dom(F_{p'})$ as otherwise the proof is a simpler version of the proof in this case.

 Let $G$ be any $\bbP_{\alpha_0}$-generic filter such that $p'\restriction \alpha_0\in G$, let $G_0=G\cap\bbP_{\alpha_0}\restriction\,\mtcl U^{\alpha_0}$, and let us note that $G_0$ is $\bbP_{\alpha_0}\restriction\,\mtcl U^{\alpha_0}$-generic over $\bold V$ by Lemma \ref{compl1}.
  Working in $N^*[G_0]$, let $E_0$ be the set of $\nu<\omega_1$ for which there is some $r\in D$ such that:
 \begin{enumerate}
 \item $r\restriction\mtcl U^{\alpha_0}\in G_0$;
 \item $\alpha_0\in \dom(F_r)$ and $F_r(\alpha_0)=\nu$.
 \end{enumerate}

 Let $\bbQ=(\name{\bbQ}_{\alpha_0})_{G_0}$ and let $E$ be the set of $\nu\in \omega_1$ such that
 \begin{itemize}
 \item $\nu \in E_0$, or else
 \item $\nu$ is incompatible in $\bbQ$ with all conditions in $E_0$.
 \end{itemize}

 $E$ is of course in $N^*[G_0]$, and it is trivially a predense subset of $\bbQ$. Also, given any $\nu_0\in E_0$, every $\nu<\omega_1$ such that $\nu\leq_{\bbQ}\nu_0$ is also in $E_0$.

Since $F_{p'}(\alpha_0)$ is $(N[G_0], \bbQ)$-generic, by Lemma \ref{cc} it is also $(N^*[G_0], \bbQ)$-generic. Hence, there is some $\nu^*\in \delta_N\cap E$ which is $\bbQ$-compatible with $F_{p'}(\alpha_0)$. Since $p'\in D$, it follows that in fact $\nu^*\in E_0$. Let us fix $r^*\in G_0$ for which there is some $r\in D$ such that $r^*=r\restriction\mtcl U^{\alpha_0}$, $\alpha_0\in \dom(F_r)$, and $F_r(\alpha_0)=\nu^*$. Let $\bbP_{\alpha_0}/G_0=\{t\in\bbP_{\alpha_0}\,:\, t\restriction\mtcl U^{\alpha_0}\in G_0\}$.  By induction hypothesis $(1)_{\alpha_0}$, $p'\restriction \alpha_0$ is $(N^*, \bbP_{\alpha_0})$-generic. Since $\bbP_{\alpha_0}\restriction\,\mtcl U^{\alpha_0}$ is a complete suborder of $\bbP_{\alpha_0}$, it follows that
\begin{enumerate}
\item $p'\restriction \mtcl U^{\alpha_0}$ is $(N^*, \bbP_{\alpha_0}\restriction\,\mtcl U^{\alpha_0})$-generic\footnote{This of course follows also from $(2)_{\alpha_0}$.} and
\item $p'\restriction \mtcl U^{\alpha_0}$ forces $p'\restriction\alpha_0$ to be $(N^*[G_0], \bbP_{\alpha_0}/G_0)$-generic.\footnote{This is true by (1) together with the fact that, since $\bbP_{\alpha_0}\restriction\,\mtcl U^{\alpha_0}$ is a complete suborder of $\bbP_{\alpha_0}$, $\bbP_{\alpha_0}$ can be represented as a two-step iteration $(\bbP_{\alpha_0}\restriction\,\mtcl U^{\alpha_0})\ast\name{\mtcl Q}$---where $\name{\mtcl Q}$ is a name for the suborder of $\bbP_{\alpha_0}$ consisting of those $t\in\bbP_{\alpha_0}$ such that $t\restriction\,\mtcl U^{\alpha_0}\in\name{G}_{\bbP_{\alpha_0}\restriction\,\mtcl U^{\alpha_0}}$---and together with the fact that a condition $(a, \name b)$ in a forcing $\mathbb A\ast\name{\mathbb B}$ is $(M, \mathbb A\ast\name{\mathbb B})$-generic, for a given model $M$, if and only if $a$ is $(M, \mathbb A)$-generic and $a$ forces $\name b$ to be $(M[\name{G}_{\mathbb A}], \name{\mathbb B})$-generic.}
\end{enumerate}

By (1), we may assume that $r^*\in N$. But then, by (2) and since $p'\restriction\alpha_0\in G$ and $r^*\in N$, we may find $r_0\in N\cap D$ such that:
\begin{itemize}
\item $r_0\restriction \alpha_0\in G$;
\item $r_0\restriction\mtcl U^{\alpha_0}=r^*$;
\item $\alpha_0\in \dom(F_{r_0})$ and $F_{r_0}(\alpha_0)=\nu^*$.
\end{itemize}

Let $p''$ be a condition in $\bbP_{\alpha_0}$ forcing the above for $r_0$ and forcing some $\nu^{**}<\omega_1$ to be a common extension of $\nu^*$ and $F_{p'}(\alpha_0)$ in $\name{\bbQ}$. Let $F=F_{p''}\cup\{(\alpha_0, \nu^{**})\}$. Then, using Lemma \ref{amalg}, we have that $(F, \Delta)$, for $$\Delta=\Delta_{p''}\cup\Delta_{p'}\cup\Delta_{r_0}\cup\{(\Psi_{N, N'}(M), \alpha)\,:\,(M, \alpha)\in\Delta_{r_0},\,(N', \alpha)\in\Delta_{p'},\,\delta_{N'}=\delta_N\}),$$ is a condition in $\bbP_\alpha$, and it extends both $p'$ and $r_0$. To see that $(F, \Delta)\in\bbP_\alpha$, the only possibly nontrivial verification is that of clause (3) in the definition of condition. For this, we note that if $(M, \alpha)\in\Delta_{r_0}$, $(N', \alpha)\in\Delta_{p'}$ is such that $\delta_{N'}=\delta_N$, and $\beta\in\mtcl U^{\alpha_0}\cap\Psi_{N, N'}(M)$, then $(\Psi_{N, N'}(M), \beta)\in\Delta_{p''}$ since $\beta=\Psi_{N', N}(\beta)\in\mtcl U^{\alpha_0}\cap M$ and hence $(M, \beta)\in\Delta_{r_0}$ (as $(M, \alpha)\in\Delta_{r_0}$). This finishes the proof in this case.

Let us now consider the case that $\alpha$ is a limit ordinal. Again, let $D\in N^*$ be an open dense subset of $\bbP_\alpha$ and let $p'\in\bbP_\alpha$ be a condition extending $p$, which we  may assume is in $D$. Let $\bar\alpha\in N\cap\alpha$ be high enough so that $\dom(F_{p'})\cap [\bar\alpha,\,\alpha)\cap N=\emptyset$. If $\cf(\alpha)=\omega$, we may assume that in fact $\dom(F_{p'})\setminus\bar\alpha=\emptyset$; and if $\cf(\alpha)>\omega$, we may assume that $M\cap [\bar\alpha, \,\alpha)\cap N=\emptyset$ for every $M\in\dom(\Delta_{p'})$ such that $\delta_{M}<\delta_N$.\footnote{Otherwise there would be, by the symmetry and finiteness of $\dom(\Delta_{p'})$, some $M\in\dom(\Delta_{p'})\cap N$ cofinal in $\alpha$, which is impossible since $\cf(\alpha)>\omega$.}
 Let also $G\in\bbP_{\bar\alpha}$ be generic over $\bold V$ and such that $p'\restriction\bar\alpha\in G$.

Suppose first that $\cf(\alpha)=\omega$. In this case, working in $N^*[G]$ we may find a condition $r\in D$ such that:
\begin{enumerate}
\item $r\restriction\bar\alpha\in G$,
\item $\dom(F_r)\sub\bar\alpha$, and
\item $\Delta_{p'}\cap N\sub\Delta_r$.
\end{enumerate}

Such an $r$ can be found, by correctness of $N^*[G]$, since the existence of such a condition is a true statement, as witnessed by $p'$, which is expressible by a sentence with parameters in $N^*[G]$ thanks to the choice of $\bar\alpha$.

By induction hypothesis $(1)_{\bar\alpha}$ we have that $p'\restriction \bar\alpha$ is $(N^*, \bbP_{\bar\alpha})$-generic. Hence we may assume that $r\in N$. Let $p''\in G$ be a condition extending $p'\restriction\bar\alpha$ and $r\restriction \bar\alpha$ and deciding $r$. Then we have that $$(F, \Delta)=(F_{p''}, \Delta_{p''}\cup\Delta_{p'}\cup\Delta_r\cup\{(\Psi_{N, N'}(M), \rho)\,:\,(M, \rho)\in\Delta_r,\,(N', \rho)\in\Delta_{p'},\,\delta_{N'}=\delta_N\})$$ is a condition in $\bbP_\alpha$. Indeed, all clauses in the definition of $\bbP_\alpha$-condition, except for (3), are straightforwardly verified for this object using Lemma \ref{amalg} and the fact that all of $p'$, $p''$ and $r$ are conditions. For clause (3) we argue very much as in the successor case. It is enough to show that if $(M, \rho)\in\Delta_r$ and $(N', \rho)\in\Delta_{p'}$ is such that $\delta_{N'}=\delta_N$, then $(\Psi_{N, N'}(M), \bar\rho)\in\Delta$ for every $\bar\rho\in\mtcl U^\rho\cap \Psi_{N, N'}(M)$. But this is immediate as then $\bar\rho=\Psi_{N', N}(\bar\rho)\in\mtcl U^\rho\cap M$ (and therefore $(M, \bar\rho)\in\Delta_r$) and also $(N', \bar\rho)\in\Delta_{p'}$ (since $\bar\rho\in\mtcl U^\rho\cap N'$). Once we know that $(F, \Delta)$ is a condition, it it trivial to see that it extends both $p'$ and $r$, which finishes the proof in this subcase.

Finally, suppose $\cf(\alpha)\geq\omega_1$. This time, working  in $N^*[G]$, we may find a condition $r\in D$ such that $r\restriction\bar\alpha\in G$. As in the previous subcase, we may assume that $r\in N$.  Let $p''\in G$ be a condition extending $p'\restriction\bar\alpha$ and $r\restriction\bar\alpha$ and deciding $r$.
 Let $(\beta_i)_{i<n}$ be the strictly increasing enumeration of $\dom(F_r)\setminus\bar\alpha$.
 We may assume that $n>0$ since otherwise we can finish as in (a simpler version of) the previous subcase.

Let $$q=(F_{p''}, \Delta_{p''}\cup\Delta_{p'}\cup\Delta_r\cup \{(\Psi_{N, N'}(M), \rho)\,:\,(M, \rho)\in\Delta_r,\,(N', \rho)\in\Delta_{p'},\,\delta_{N'}=\delta_N\}),$$ and let us note that $q$ is a condition for the same reasons as for the object $(F, \Delta)$ constructed in the $\cf(\alpha)=\omega$ case. 
We now build a certain decreasing sequence $(q_i^+)_{i<n}$ of conditions extending $q\restriction\beta_0$. For this, at step $i$ of the construction we first fix a condition $q_i$ in $\bbP_{\beta_i}$ such that
\begin{enumerate}
\item $q_i$ extends $q\restriction\beta_i$,
\item $q_i$ extends $q^+_{i-1}$ if $i>0$, and such that
\item for some $\nu_i<\omega_1$, $q_i$ forces $\nu_i$ to be $(M[\name{G}_{\bbP_{\beta_i}\restriction\,\mtcl U^{\beta_i}}],  \name{\bbQ}_{\beta_i})$-generic for every $M$ such that $(M, \beta_i+1)\in\Delta_{p'}$.
\end{enumerate}

Then we let 
$q_i^+=(F_{q_i}\cup\{(\beta_i, \nu_i)\}, \Delta_{q_i}\cup(\Delta_q\restriction\beta_{i+1}))$, 
where $\beta_n=\alpha$, which will clearly be a $\bbP_{\beta_{i+1}}$-condition.

It is trivial to find $q_i$ satisfying (1) and (2). The reason $q_i$ can be found in such a way that it satisfies (3) as well is that $\mtcl M=\{M[\name{G}_{\bbP_{\beta_i}\restriction\,\mtcl U^{\beta_i}}]\,:\,(M, \beta_i+1)\in\Delta_{p'},\,\delta_M\geq\delta_N\}$ is forced to be a symmetric system by Lemma \ref{same-subsets-o1-ext} and therefore we can apply an argument as in the proof of Lemma \ref{symm-small-proper} with that set. More specifically, working in $\bold{V}^{\bbP_{\beta_i}\restriction q_i}$, let $(N^i_k)_{k\leq l}$ be a $\sub$-maximal $\in$-chain of members of $\mtcl M$. Noting then that $F_r(\beta_i)\in N^i_0$, we build a decreasing sequence $(\nu^i_k)_{k\leq l+1}$ of conditions in $\name{\bbQ}_{\beta_i}$ by setting $\nu^i_0=F_r(\beta_i)$ and letting $\nu^i_{k+1}$ be, for each $k\leq l$, some $(N^i_k, \name{\bbQ}_{\beta_i})$-generic condition in $N^i_{k+1}$ extending $\nu^i_k$ (where $N^i_{l+1}=\cH(\kappa)$). The reason $\nu^i_{k+1}$ can be found is that $N_k^i$ is the intersection with $\cH(\kappa)$ of an elementary submodel $N^*$ of $\cH(\kappa^{+(1+\beta_i+1)})$ containing a $\bbP_{\beta_i}\restriction\,\mtcl U^{\beta_i}$-name $\name E$ for a club of $[\cH(\kappa)]^{\aleph_0}$ in $\bold V$ witnessing that $\name{\bbQ}_{\beta_i}$ is forced, in $\bbP_{\beta_i}\restriction\,\mtcl U^{\beta_i}$, to be $(\name{G}_{\bbP_{\beta_i}\restriction\,\mtcl U^{\beta_i}}, \bbP_{\beta_i})$-proper in $\bold V[\name{G}_{\bbP_{\beta_i}\restriction\,\mtcl U^{\beta_i}}]$. We then have, by $(2)_{\beta_i}$, that $p'\restriction\mtcl U^{\beta_i}$ is $(N^*, \bbP_{\beta_i}\restriction\,\mtcl U^{\beta_i})$-generic and therefore it forces $N$ to be in $\name E$. Hence we can find $\nu^i_{k+1}$ as desired.


Then $q^+_{n-1}$ is a condition in $\bbP_\alpha$ extending both $p'$ and $r$, which concludes the proof of $(1)_\alpha$ in this subcase.

It is perhaps worth mentioning here that thanks to the above choice of $\bar\alpha$, the subsequent argument in $\bold{V}[G]$ succeeded in producing a condition $r$ which could be amalgamated with $p'$ into a stronger condition. The point is that, in the $\cf(\alpha)\geq\omega_1$ case, we managed to make sure that no points $\beta$ in the domain of $F_r$ above $\bar\alpha$ be in $\dom(F_{p'})$ or in any relevant models $M$ of height less than $N$---otherwise, the working part $F_r(\beta)$ could have conflicted with $F_{p'}(\beta)$, if $\beta\in\dom(F_{p'})$, or with the genericity requirement for some $M$ as above if $\beta\in M$ for some such $M$. And of course, in the case $\cf(\alpha)=\omega$, by the choice of $\bar\alpha$ we were able to ensure that $\dom(F_r)\setminus\bar\alpha$ be actually empty, so none of the problems we just mentioned could possibly arise. 
\end{proof}

We can now prove the following.

\begin{lemma}\label{properness1}
Let $\alpha<\kappa$,  $p\in\bbP_\alpha$, and let $(N, \alpha)$ be a model with marker such that $p\in N$.
Then there is an extension $p^*\in\bbP_\alpha$ of $p$ such that $(N, \rho)\in \Delta_{p^*}$ for all $\rho\in N\cap (\alpha+1)$.
\end{lemma}

\begin{proof}
We prove this by induction on $\alpha$. We may assume that $\dom(F_p)\neq\emptyset$ as otherwise we may simply let $p^*=(\emptyset, \Delta_p\cup\{(N, \rho)\,:\,\rho\in N\cap (\alpha+1)\})$.

Let $\bar\alpha=\max(\dom(F_p)\cap\alpha)$. By induction hypothesis we know that there is an extension $\bar p$ of $p\restriction\bar\alpha$ such that  $(N, \rho)\in\Delta_{\bar p}$ for each  $\rho\in N\cap (\bar\alpha+1)$.
 Let $N^*\prec \cH(\kappa^{+(1+\bar\alpha+1)})$ be a countable model such that $\phi$, $\triangleleft\in N^*$ and $N=N^*\cap\cH(\kappa)$. We argue now as in the proof of Lemma \ref{properness} $(1)_\alpha$ for the $\cf(\alpha)\geq\omega_1$ subcase. We note that $N^*$ contains a $\bbP_{\bar\alpha}\restriction\,\mtcl U^{\bar\alpha}$-name $\name E$ for a club of $[\cH(\kappa)]^{\aleph_0}$ in $\bold V$ witnessing that $\name{\bbQ}_{\bar\alpha}$ is a $(\name{G}_{\bbP_{\bar\alpha}\restriction\,\mtcl U^{\bar\alpha}}, \bbP_{\bar\alpha})$-proper poset. By Lemma \ref{properness} $(2)_{\bar\alpha}$, $\bar p\restriction\mtcl U^{\bar\alpha}$ is $(N^*, \bbP_{\bar\alpha}\restriction\,\mtcl U^{\bar\alpha})$-generic. In particular, $\bar p$ forces that $N=N^*\cap\cH(\kappa)\in \name E$. Since $(N, \rho)\in \Delta_{\bar p}$ for all $\rho\in N\cap (\bar\alpha+1)$, we may then extend $\bar p\restriction\mtcl U^{\bar\alpha}$ to a condition $p'\in\bbP_{\bar\alpha}\restriction\,\mtcl U^{\bar\alpha}$ for which there is some $\nu$ such that $p'$ forces $\nu$ to be an $(N[\name{G}_{\bbP_{\bar\alpha}\restriction\,\mtcl U^{\bar\alpha}}], \name{\bbQ}_{\bar\alpha})$-generic condition in $\name{\bbQ}_{\bar\alpha}$ stronger than $F_p(\bar\alpha)$.

Finally, let $p^*=(F_*, \Delta_*)$, where

 \begin{itemize}

\item $F_*=F_{p'}\cup (F_{\bar p}\restriction \bar\alpha\setminus\mtcl U^{\bar\alpha})\cup\{(\bar\alpha, \nu)\}$ and
\item $\Delta_*= \Delta_{p'}\cup\Delta_{\bar p}\cup\Delta_p\cup\{(N, \rho)\,:\, \rho\in N\cap (\alpha+1)\}$.
\end{itemize}

Then $p^*$ is an extension of $p$ as desired.
\end{proof}

Lemmas \ref{properness} and \ref{properness1}, together with Lemma \ref{cc} (for the case $\alpha=\kappa$), yield the following corollary.

\begin{corollary}\label{cor-prop} For each $\alpha\leq\kappa$, $\bbP_\alpha$ is proper.
\end{corollary}

The following lemma shows that $\bbP_\kappa$ produces a model of the relevant forcing axiom.

\begin{lemma}\label{pfaomega_1}
$\bbP_\kappa$ forces $\PFA_{{<}\kappa}(\aleph_1)$.
\end{lemma}

\begin{proof}
Let $\name{\bbQ}$ be a $\bbP_\kappa$-name for a proper forcing on $\omega_1^{\bold V[\name{G}_{\bbP_\kappa}]}$ ($=\omega_1^{\bold V}$) and let $\{\name{D}_i\,:\,i<\lambda\}$ be, for some $\lambda<\kappa$, $\bbP_\kappa$-names for dense subsets of $\name{\bbQ}$. By Lemma \ref{cc} and the fact that $\kappa\geq\omega_2$ is regular, there is some $\alpha_0<\kappa$ such that $\name{\bbQ}$ and $\name{D}_i$, for all $i<\lambda$, are $\bbP_\alpha$-names for all $\alpha$, $\alpha_0\leq\alpha<\kappa$. Letting $\langle A_i\,:\, i<\omega_1\rangle$ be an $\omega_1$-sequence of antichains of $\bbP_{\alpha_0}$ such that $\bigcup_{i<\omega_1}\{i\}\times A_i$ is a nice $\bbP_{\alpha_0}$-name for a subset of $\omega_1$ canonically encoding $\name{\bbQ}$, we may find $\alpha\geq\alpha_0$ such that $\phi(\alpha)=\langle A_i\,:\,i<\omega_1\rangle$.

\begin{claim}
$\bbP_\alpha\restriction\,\mtcl U^\alpha$ forces that $\name{\bbQ}$ is $(\name{G}_{\bbP_\alpha\restriction\,\mtcl U^\alpha}, \bbP_\alpha)$-proper.
\end{claim}

\begin{proof}
Let $f:[\cH(\kappa)]^{{<}\omega}\rightarrow\cH(\kappa)$ be a function with the property that for every countable $N\sub\cH(\kappa)$, if $f``[N]^{{<}\omega}\sub N$, then $N=N^*\cap \cH(\kappa)$ for some countable $N^*\prec \cH(\kappa^{+(\kappa+1)})$ such that $\phi$, $\langle\bbP_\gamma\,:\,\gamma<\kappa\rangle\in N^*$.
Let $G$ be $\bbP_\alpha\restriction\,\mtcl U^\alpha$-generic, let $\bbQ=\name{\bbQ}_G$, let $N$ be such that  $f``[N]^{{<}\omega}\sub N$ and such that,  for some $p\in\bbP_\alpha$ with $p\restriction\mtcl U^\alpha\in G$, $(N, \rho)\in \Delta_p$ for all $\rho\in N\cap \mtcl U^\alpha$, and let $\nu\in\delta_N$. It will suffice to show that there is some $\nu^*\leq_{\bbQ}\nu$ which is $(N[G], \bbQ)$-generic.

Suppose, for a contradiction, that this is not the case. By extending $p$ if necessary, we may then assume that  $p\restriction\mtcl U^\alpha$ forces that there is no  $\nu^*\leq_{\name{\bbQ}}\nu$ which is $(N[\name{G}_{\bbP_\alpha\restriction\,\mtcl U^\alpha}], \name{\bbQ})$-generic. Let now $\nu^*\in\omega_1$ and $q\leq_\kappa p$ be such that $(N, \beta)\in \Delta_q$ for each $\beta\in N\cap\kappa$ and such that $q$ forces that $\nu^*\leq_{\name{\bbQ}}\nu$ and that $\nu^*$ is $(N[\name{G}_{\bbP_\kappa}], \name{\bbQ})$-generic. Such a $q$ exists by the same argument as in the proof of Lemma \ref{properness1}, together with the choice of $\name{\bbQ}$ as a name for a proper forcing. We have that $q$ forces $\delta_{N[\name{G}_{\bbP_\kappa}]}=\delta_{N[\name{G}_{\bbP_\alpha\restriction\,\mtcl U^\alpha}]}=\delta_N$ by Lemma \ref{properness} $(1)_\beta$, for $\beta\in N\cap\kappa$, together with Lemma \ref{cc}. Since $\bbP_\alpha\restriction\,\mtcl U^\alpha\lessdot \bbP_\kappa$, it then follows that $q\restriction\mtcl U^\alpha$ forces $\nu^*$ to be $(N[\name{G}_{\bbP_\alpha\restriction\,\mtcl U^\alpha}], \name{\bbQ}_\alpha)$-generic. But that is a contradiction since $q\restriction\mtcl U^\alpha$ extends $p\restriction\mtcl U^\alpha$ in $\bbP_\alpha\restriction\,\mtcl U^\alpha$.
\end{proof}

By the claim, together with the choice of $\langle A_i\,:\, i<\omega_1\rangle$, we have that $\name{\bbQ}_\alpha=\name{\bbQ}$. This finishes the proof of the lemma since then $\bbP_{\alpha+1}$ forces $$H=\{\nu<\omega_1\,:\, \nu\in F_p(\alpha)\text{ for some }p\in\name{G}_{\bbP_{\alpha+1}}\text{ with }\alpha\in\dom(F_p)\}$$ to be a generic filter for $\name{\bbQ}$ over $\bold V[\name{G}_{\bbP_\alpha}]$ and hence such that $H\cap \name{D}_i\neq\emptyset$ for each $i<\lambda$. To see this, it suffices to argue that we may extend any condition $p\in\bbP_\kappa$ to a condition $p^*$ with $\alpha\in\dom(F_{p^*})$. But this is possible by an argument exactly as the one dealing with a fixed $\beta_i$ in the $\cf(\alpha)\geq\omega_1$-subcase of the proof of Lemma \ref{properness} $(1)_\alpha$.
\end{proof}

Finally, we prove that $\bbP_\kappa$ forces the right cardinal arithmetic.

\begin{lemma}\label{arithm}
$\bbP_\kappa$ forces $2^{\aleph_0}=2^{\aleph_1}=\kappa$.
\end{lemma}

\begin{proof}
 $\Vdash_{\bbP_\kappa}2^{\aleph_0}\geq\kappa$ follows, for example, from $\Vdash_{\bbP_\kappa}\FA(\{\text{Cohen}\})_{{<}\kappa}$.
$\Vdash_{\bbP_\kappa}2^{\aleph_1}\leq\kappa$ follows from the fact that there are $(\kappa^{\aleph_1})^{\aleph_1}=\kappa$ nice $\bbP_\kappa$-names for subsets of $\omega_1$.
\end{proof}

Lemma \ref{arithm} concludes the proof of the theorem.

\section{A $\Pi_2$-rich model of $\PFA(\omega_1)_\kappa$ for large $\kappa$}\label{more}
Given a forcing notion $\bbQ$ and a collection $\mtcl N$ of countable models, let us say that a condition $q^*\in\bbQ$ is \emph{$\mtcl N$-symmetric} in case for all $q'\in\bbQ$ such that $q'\leq_{\bbQ}q^*$, all $M_0$, $M_1\in\mtcl N$ such that $M_0\cong M_1$, and all $q\in M_0\cap\bbQ$, if $q'\leq_{\bbQ} q$, then also $q'\leq_{\bbQ}\Psi_{M_0, M_1}(q)$.

We note (and will sometimes use implicitly) that any condition stronger than a given $\mtcl N$-symmetric condition is itself $\mtcl N$-symmetric.

Throughout the rest of the paper, we will be considering the notion of a condition $q^*$ in a forcing notion $\bbQ$ being $(N, \bbQ)$-generic for a model $N$, also in cases where $\bbQ\notin N$. In either case, the definition we will be working with is the usual one, formulated in terms of maximal antichains; i.e., we say that $q^*$ is \emph{$(N, \bbQ)$-generic} if for every maximal antichain $A$ of $\bbQ$ such that $A\in N$ and every extension $q'\leq_\bbQ q^*$ there is some $r\in A\cap N$ compatible with $q'$.

Given a forcing notion $\bbQ$ and a cardinal $\theta$ such that $\bbQ\sub\cH(\theta)$,\footnote{In other words, such that $|\TC(\bbQ)|\leq 2^{{<}\theta}$, where $\TC(\bbQ)$ is the transitive closure of $\bbQ$.} we say that $\bbQ$ is \emph{symmetrically proper relative to $\cH(\theta)$} in case there is a predicate $P\sub\cH(\theta)$ with the property that for every finite $P$-symmetric system $\mtcl N$ of countable elementary submodels of $\cH(\theta)$ containing $\bbQ$, if $q\in\bbQ\cap M$ for some $M\in\mtcl N$ of minimal height within $\mtcl N$, then there is a condition extending $q$ which is $\mtcl N$-symmetric and is $(M, \bbQ)$-generic for every $M\in\mtcl N$.

\begin{proposition} Given a forcing notion $\bbQ$ and cardinals $\theta_0\leq\theta_1$ such that $\theta_0\geq\omega_2$, $\bbQ\sub\cH(\theta_0)$, and such that every maximal antichain of $\bbQ$ belongs to $\cH(\theta_0)$,\footnote{Of course, since every condition in $\bbQ$ belongs to some maximal antichain of $\bbQ$, we have that $\bbQ\sub\cH(\theta_0)$.} if $\bbQ$ is symmetrically proper with respect to $\cH(\theta_0)$, then it is also symmetrically proper with respect to $\cH(\theta_1)$. \end{proposition}

\begin{proof}
If $P\sub\cH(\theta_0)$ witnesses the symmetric properness of $\bbQ$ relative to $\cH(\theta_0)$, then any predicate $P^*\sub\cH(\theta_1)$ such that $P$ and $\theta_0$ are both definable in $(\cH(\theta_1); \in, P^*)$ witnesses the symmetric properness of $\bbQ$ relative to $\cH(\theta_1)$. For this, it suffices to note that if $\mtcl N$ is a $P^*$-symmetric system, then $\mtcl N\restriction\cH(\theta_0):=\{N\cap\cH(\theta_0)\,:\,N\in\mtcl N\}$ is a $P$-symmetric system. This is enough since, by our hypothesis that every maximal antichain of $\bbQ$ belongs to $\cH(\theta_0)$, we have that for every $N\in\mtcl N$, a condition is $(N, \bbQ)$-generic if and only if it is $(N\cap\cH(\theta_0), \bbQ)$-generic.
\end{proof}

We define the property of being symmetrically proper as the strongest one of these parameterized properties (whenever this makes sense); in other words, a forcing notion $\bbQ$ is \emph{symmetrically proper} if the least cardinal $\theta$ such that $|\TC(\bbQ)|\leq 2^{{<}\theta}$ is at least $\omega_2$ and $\bbQ$ is symmetrically proper relative to $\cH(\theta)$.

Given a cardinal $\mu$, a poset $\bbQ$ is \emph{$\mu$-linked} in case there is a decomposition $\bbQ=\bigcup_{\xi<\mu}A_\xi$ such that every $A_\xi$ is linked; i.e, any two conditions in $A_\xi$ are compatible in $\bbQ$.

\begin{proposition}\label{two-step} Both the class of symmetrically proper forcings and the class of $\mu$-linked forcings, for any fixed infinite cardinal $\mu$, are closed under two-step iterations.
\end{proposition}

\begin{proof}
Let us start by proving the first assertion. Suppose $\bbP$ is a symmetrically proper forcing, let $\theta_0$ be least cardinal such that $|\TC(\bbP)|\leq 2^{{<}\theta_0}$ , and let $P_0\sub\cH(\theta_0)$ be a predicate witnessing the symmetric properness of $\bbP$. Let also $\name{\bbQ}$ be a $\bbP$-name for a symmetrically proper forcing, let $\name{\theta}_1$ be a $\bbP$-name for the least cardinal $\theta$ such that $|\TC(\name{\bbQ})|\leq 2^{{<}\theta}$, let $\name{P}_1$ be a $\bbP$-name for a predicate of $\cH(\name{\theta}_1)$ witnessing the symmetric properness of $\name\bbQ$, let $\theta_2$ be the least cardinal such that $|\TC(\bbP\ast\name{\bbQ})|\leq 2^{{<}\theta_2}$, and let us note that $\theta_2\geq\theta_0$ and $\Vdash_\bbP\theta_2\geq\name{\theta}_1$. Let $P_2\sub\cH(\theta_2)$ be such that $P_0$, $\bbP\ast\name{\bbQ}$, and the set of $(p, \name x)$, where $p\in\bbP$, $\name x\in\cH(\theta)$, and $p\Vdash_\bbP\name x\in\name{P}_1$, are all definable in $(\cH(\theta_2); \in, P_2)$.  

Let $\mtcl N$ be a finite $P_2$-symmetric system and let $(p_0, \name{q}_0)\in N_0$ for some $N_0\in\mtcl N$ of minimal height within $\mtcl N$. We will find an extension $(p^*, \name q^*)\in\bbP\ast\name{\bbQ}$ of $(p_0, \name q_0)$ which is $\mtcl N$-symmetric and $(N, \bbP\ast\name{\bbQ})$-generic for every $N\in\mtcl N$.

Since $P$ is definable in $(\cH(\theta_2); \in, P_2)$ and $\theta_2\geq\theta_0$, we may fix an extension $p^*$ of $p_0$ in $\bbP$ which is $\mtcl N$-symmetric and $(N, \bbP)$-generic for every $N\in\mtcl N$.

\begin{claim}\label{cl3}
$p^*$ forces in $\bbP$ that $\mtcl N^{\name G_\bbP}=\{N[\name G_\bbP]\,:\,N\in \mtcl N\}$ is a $\name{P}_1$-symmetric system.
\end{claim}

\begin{proof}
Thanks to the choice of $P_2$, $N[\name{G}_\bbP]\cap\cH(\name{\theta}_1)$ is forced to be an elementary submodel of $(\cH(\name{\theta}_1); \in, \name{P}_1)$ for each $N\in\mtcl N$. Hence, it suffices to prove that $p^*$ forces $N^0[\name{G}_{\bbP}]\cong N^1[\name{G}_{\bbP}]$ for all $N^0$, $N^1\in\mtcl N$ of the same height. But by essentially the same argument as in the proof of Lemma \ref{same-subsets-o1-ext}, using the fact that $p^*$ is $\mtcl N$-symmetric in $\bbP$, we can see that $p^*$ forces, for all $N^0$ and $N^1$ as above, that the function sending $\name{X}_{\name{G}_{\bbP}}$, for a $\bbP$-name $\name X\in N^0$, to $\Psi_{N^0, N^1}(\name X)_{\name{G}_{\bbP}}$ is an isomorphism between $N^0[\name{G}_{\bbP}]$ and $N^1[\name{G}_{\bbP}]$.
\end{proof}

By the claim, together with the fact that $\name{\bbQ}$ is forced to be symmetrically proper as witnessed by $\name{P}_1$, the fact that $p^*\Vdash_{\bbP}\name q\in\bbQ$, and the fact that $N_0[\name{G}_\bbP]$ is of minimal height within $\mtcl N^{\name G_\bbP}$, we may fix a $\bbP$-name $\name{q}^*$ such that $p^*$ forces $\name{q}^*$ to be a condition in $\name{\bbQ}$ extending $\name{q}_0$ which is $\mtcl N^{\name G_\bbP}$-symmetric and $(N[\name{G}_{\bbP}], \name\bbQ)$-generic for each $N\in\mtcl N$. Then $(p^*, \name{q}^*)$ is a condition in $\bbP\ast\name{\bbQ}$ extending $(p_0, \name q_0)$ which is $(N, \bbP\ast\name{Q})$-generic for each $N\in\mtcl N$, and it is easy to see that it is also $\mtcl N$-symmetric. For this, suppose $(p', \name q')\leq_{\bbP\ast\name{\bbQ}}(p^*, \name q^*)$, $N^0$, $N^1\in\mtcl N$ are such that $\delta_{N^0}=\delta_{N^1}$, and $(p, \name q)\in N^0\cap\bbP\ast\name{\bbQ}$ is weaker than $(p', \name q')$. We then have that also $p'\leq_\bbP\Psi_{N^0,N^1}(p)$ and that $p'$ forces that $\name{q}'\leq_{\name{\bbQ}}\Psi_{N^0[\name{G}_{\name{\bbQ}}], N^1[\name{G}_{\name{\bbQ}}]}(\name q)=\Psi_{N^0, N^1}(\name q)$, where the equality follows from the proof of Claim \ref{cl3}. In other words, $(p', \name{q}')\leq_{\bbP\ast\name{Q}}(\Psi_{N^0, N^1}(p), \Psi_{N^0, N^1}(\name q))=\Psi_{N^0, N^1}((p, \name q))$. This finishes the proof of the first assertion.

As to the second assertion, it suffices to observe that if $\{A_\xi\,:\,\xi<\mu\}$ is a decomposition of $\bbP$ witnessing its $\mu$-linkedness, $(\name{B}_\zeta\,:\,\zeta<\mu)$ is a sequence of $\bbP$-names, and $\bbP$ forces $\{\name{B}_\zeta\,:\,\zeta<\mu\}$ to be a decomposition of a forcing $\name{\bbQ}$ witnessing its $\mu$-linkedness, then letting $C_{\xi, \zeta}$, for $(\xi, \zeta)\in\mu\times\mu$, be the collection of conditions of the form $(p, \name q)$, where $p\in A_\xi$ and $p\Vdash_\bbP\name q\in\name{B}_\zeta$, we have that $\{C_{\xi, \zeta}\,:\,(\xi, \zeta)\in\mu\times\mu\}$ is a decomposition of $\bbP\ast\name{\bbQ}$ witnessing its $\mu$-linkedness.
\end{proof}

A sufficient, and useful, condition for a poset to be $\omega_1$-linked is that it be a P\v{r}\'{i}kr\'{y}-type partial order with stems in $\omega_1$. In general, we say that a poset $\bbQ$ is a \emph{P\v{r}\'{i}kr\'{y}-type} partial order in case there is a set $\Res(\bbQ)$ such that:
\begin{enumerate}
\item $\bbQ$ is a partial order with conditions being ordered pairs $(s, A)$ such that $A\in\Res(\bbQ)$;
\item for all $A_0$, $A_1\in \Res(\bbQ)$, $A_0\cap A_1\in\Res(\bbQ)$;
\item for every $(s, A_0)\in\bbQ$ and every $A_1\in\Res(\bbQ)$, if $A_1\sub A_0$, then $(s, A_1)$ is a condition in $\bbQ$ extending $(s, A_0)$;
\end{enumerate}

In the above situation we will sometimes refer to $s$ as the \emph{stem} of $(s, A)$ and to $A$ as its \emph{reservoir}. Given a set $X$, we will say that a P\v{r}\'{i}kr\'{y}-type partial order $\bbQ$ has \emph{stems in $X$} if for all $(s, A)\in\bbQ$, $s\in X$.

Let \emph{Local $\CH$} be the statement that every set in $\cH(\aleph_2)$ belongs to some ground model satisfying $\CH$.

In this section we will be mostly concerned with $\omega_1$-linked symmetrically proper forcing notions.\footnote{The conjunction of these two properties is of a very similar flavour to the property of having the $\aleph_2$-p.i.c. At the moment it is not clear to us what the exact relationship between our class and the class of $\aleph_2$-p.i.c.\ forcings is.} In particular, we will be interested in the following forcing axiom-like principle.

\begin{definition}
Given a cardinal $\mu$, \emph{$\omega_1$-linked symmetric $\BPFA(\mu)$ from ground models of $\CH$}, which we will also denote
$\CH$-$\omega_1$\text{-linked-Symm-}$\BPFA(\mu)$, is the conjunction of the following two statements.

\begin{enumerate}

\item Local $\CH$

\item Let $a\in \cH(\aleph_2)$ and let $\varphi(x, y)$ be a $\Sigma_0$ formula in the language of set theory. Suppose for every ground model $M$, if $a\in M$ and $M\models\CH$, then it holds in $M$ that there is an $\omega_1$-linked symmetrically proper forcing notion $\bbQ\sub\cH(\mu)^M$ such that $\bbQ$ forces $\cH(\aleph_2)\models \exists y\varphi(a, y)$. Then $\cH(\aleph_2)\models\exists y\varphi(a, y)$.
 \end{enumerate}
\end{definition}

Note that, by the first order definability of the generic multiverse (\cite{Laver}), our principle $\CH$--$\omega_1$\text{-linked-Symm-}$\BPFA(\mu)$, and hence also Local $\CH$, are first order statements in the language of set theory.

Obviously, $\CH$--$\omega_1$\text{-linked-Symm-}$\BPFA(\mu')$ implies $\CH$--$\omega_1$\text{-linked-Symm-}$\BPFA(\mu)$ for all cardinals $\mu<\mu'$.

The following is the main theorem in this section.

\begin{theorem}\label{thm2}
Assume $\GCH$. Let $\kappa\geq\aleph_2$ be a regular cardinal. Then there is an $\aleph_2$-Knaster proper partial order $\bbP$ forcing the following statements.
\begin{enumerate}
\item $2^{\aleph_0}=2^{\aleph_1}=\kappa$
\item $\PFA(\omega_1)_{{<}\kappa}$
\item $\CH$--$\omega_1$\text{-linked-Symm-}$\BPFA(\kappa)$
\end{enumerate}
\end{theorem}

In the next subsection we will prove this theorem and in Subsection \ref{applications} we will give some applications of $\CH$--$\omega_1$\text{-linked-Symm-}$\BPFA(\omega_2)$.

\subsection{Proving Theorem \ref{thm2}}\label{provingthm2}

Let us assume $\GCH$ and let $\kappa\geq\aleph_2$ be a regular cardinal. The forcing $\bbP$ witnessing Theorem \ref{thm2} will be obtained by a modification of the construction from the proof of Theorem \ref{main}, which we will also refer to as $\langle\bbP_\alpha\,:\,\alpha\leq\kappa\rangle$. In fact, we will recycle much of the notation from that proof. Also, most of the verifications of the relevant points will be the same as in the proof of Theorem \ref{main}, so we will only give details of those arguments which are new.

Let $\Fml_{\Sigma_0}(x, y)$ be the set of $\Sigma_0$ formulas in the language of set theory with free variables among $x$, $y$. There are four differences in the present construction with respect to the one from the proof of Theorem \ref{main}:

Given $\alpha<\kappa$, and assuming $\bbP_\alpha$ has been defined, we define $\mtcl U^\alpha$
 by letting
$\mtcl U^\alpha
=\emptyset$ if $\phi(\alpha)$ is not a pair of the form $(p_\alpha, \langle A^\alpha_i\,:\, i<\omega_1\rangle)$, where $p_\alpha\in \{0\}\cup\Fml_{\Sigma_0}(x, y)$ and where $\langle A^\alpha_i\,:\, i<\omega_1\rangle$ is a sequence of antichains of $\bbP_\alpha$, all of them of size $\leq\aleph_1$, and, in the other case, letting $\mtcl U^\alpha$
be defined in exactly the same way as in the proof of Theorem \ref{main}, starting of course from $\langle A^\alpha_i\,:\,i<\omega_1\rangle$.

The second difference is in the definition of $\name{\bbQ}_\alpha$. This is now a $\bbP_\alpha\restriction\,\mtcl U^\alpha$-name as follows:

\begin{itemize}
\item If $\phi(\alpha)$ is of the form $(0, \langle A^\alpha_i\,:\, i<\omega_1\rangle)$, where $\langle A^\alpha_i\,:\, i<\omega_1\rangle$ is a sequence of antichains of $\bbP_\alpha$, all of them of size $\leq\aleph_1$, such that $\bigcup_{i<\omega_1}\{i\}\times A^\alpha_i$, viewed as a nice $\bbP_\alpha\restriction\,\mtcl U^\alpha$-name for a subset of $\omega_1$, canonically encodes a forcing notion on $\omega_1^{\bold V}$ which $\bbP_\alpha\restriction\,\mtcl U^\alpha$ forces to be $(\name{G}_{\bbP_\alpha\restriction\,\mtcl U^\alpha}, \bbP_\alpha)$-proper, then $\name{\bbQ}_\alpha$ is a $\bbP_\alpha\restriction\,\mtcl U^\alpha$-name for this forcing notion.

\item If $\phi(\alpha)$ is of the form $(\varphi(x, y), \langle A^\alpha_i\,:\, i<\omega_1\rangle)$, with $\varphi(x, y)\in\Fml_{\Sigma_0}(x, y)$ and $\langle A^\alpha_i\,:\, i<\omega_1\rangle$ a sequence of antichains of $\bbP_\alpha$, all of them of size $\leq\aleph_1$, then $\name{\bbQ}_\alpha$ is the $\triangleleft$-least $\bbP_\alpha\restriction\,\mtcl U^\alpha$-name $\name{\bbQ}$ in $\cH(\kappa^+)$ which is forced to have the following property:

\begin{itemize}

\item Suppose $\name a_\alpha=\bigcup_{i<\omega_1}\{i\}\times A^\alpha_i$, viewed as a nice $\bbP_\alpha\restriction\,\mtcl U^\alpha$-name for a subset of $\omega_1$, is such that there is an $\omega_1$-linked symmetrically proper partial order $\bbQ\sub\kappa$ forcing $\exists y\varphi(y, \name a_\alpha)$. Then $\name{\bbQ}$ is such a partial order.
 \item In the other case, $\name{\bbQ}$ is trivial forcing $\{\emptyset\}$.
 \end{itemize}

\item In the remaining case, $\name{\bbQ}_\alpha$ is a $\bbP_\alpha\restriction\,\mtcl U^\alpha$-name for trivial forcing $\{\emptyset\}$.
\end{itemize}

The third difference is that now clause $(1)$ in the definition of $\bbP_\alpha$ gets replaced with the following.

\begin{enumerate}

\item[$(1)^*$] $F_p$ is a finite function with $\dom(F_p)\sub\alpha$ and such that for each $\beta\in \dom(F_p)$, $F_p(\beta)\in\kappa$.
\end{enumerate}

In the definition of the extension relation $\leq_\alpha$ we now of course require that if $p_0$, $p_1\in\bbP_\alpha$, $p_1\leq_\alpha p_0$, and $\beta\in\dom(F_{p_0})$, then $\beta\in\dom(F_{p_1})$ and $p_1\restriction\mtcl U^\beta$ forces in $\bbP_\beta\restriction\,\mtcl U^\beta$ that $F_{p_1}(\beta)\leq_{\name{\bbQ}_\beta}F_{p_0}(\beta)$.

\begin{remark}
By Lemma \ref{symm-small-proper}, every proper poset $\bbQ$ on $\omega_1$ is a dense suborder of a symmetrically proper forcing $\bbQ^*$\footnote{If $\CH$ holds, then obtain $\bbQ^*$ by simply putting $\aleph_2$-many conditions above the weakest condition of $\bbQ$.} (and it is obviously $\aleph_2$-c.c.). It does not follow from this, though, that the first two bullet points in the above specification of $\name{\bbQ}_\alpha$ could have been merged into one. The reason is that the forcings of size $\aleph_1$ we pick in our construction need not be proper in the relevant $\bbP_\alpha\restriction\,\mtcl U^\alpha$-extension, but merely forced to be $(\name{G}_{\bbP_\alpha\restriction\,\mtcl U^\alpha}, \bbP_\alpha)$-proper (cf.\ the construction for Theorem \ref{main}). Actually, the merging could have been carried out provided we had made a similar move when considering $\omega_1$-linked\ symmetrically proper forcings in the construction, in the sense of replacing symmetric properness with a similar version of this property involving only collections of extensions of models coming from suitable side conditions. However, in the interest of keeping definitions reasonably simple,\footnote{Cf.\ the above description.} we have chosen not to make this move and instead keep the first two subcases in the definition of $\name{\bbQ}_\alpha$ apart.
\end{remark}

The last difference with respect to the construction from Theorem \ref{main} is that now clause (4) in the original construction gets supplemented with the following symmetry clause.

\begin{itemize}
\item[(5)] For each $\beta\in\dom(F_p)$, $p\restriction \mtcl U^\beta$ forces in $\bbP_\beta\restriction\,\mtcl U^\beta$ that $F_p(\beta)$ is $\name{\mtcl N}^p_\beta$-symmetric (in $\name{\bbQ}_\beta$) for $$\name{\mtcl N}^p_\beta=\{N[\name{G}_{\bbP_\beta\restriction\,\mtcl U^\beta}]\,:\,(N, \beta+1)\in \Delta_p\}$$
\end{itemize}

This completes the specification of the present forcing construction.

The proofs of the corresponding versions of Lemmas \ref{compl0} and \ref{compl1} are exactly the same as the proofs of the original lemmas, and the corresponding version of Lemma \ref{cc} is also almost the same. The only difference is that, in the situation of that proof, letting $\alpha_0<\ldots <\alpha_{n-1}$ be the ordinals in $\dom(F_{p_{i_0}})\cap\dom(F_{p_{i_1}})$, which we may assume is a nonempty sequence, we recursively build a decreasing sequence $(q_i)_{i<n}$ of conditions extending $((F_{p_{i_0}}\cup F_{p_{i_1}})\restriction\alpha_0, (\Delta_{p_{i_0}}\cup\Delta_{p_{i_1}})\restriction\alpha_0)$. At each step $i$, $q_i$ is a condition in $\bbP_{\alpha_i}$ extending $q_{i-1}$ (if $i>0$) and such that $F_{q_i}(\alpha_i)$ is a name for some condition in $\name{\bbQ}_{\alpha_i}$ forced to extend both $F_{p_{i_0}}(\alpha_i)$ and $F_{p_{i_1}}(\alpha_i)$. The point is that, thanks to the fact that $\mtcl M_{i_0}\cong\mtcl M_{i_1}$, $F_{p_{i_0}}(\alpha_i)$ and  $F_{p_{i_1}}(\alpha_i)$ are forced by $q_{i-1}$ (if $i>0$) to belong to the same piece in some canonically fixed decomposition $\{A_\xi\,:\,\xi<\omega_1\}$ of $\name{\bbQ}_{\alpha_i}$ witnessing the $\omega_1$-linkedness of this poset---e.g., we can use the $\triangleleft$-least $\bbP_{\alpha_1}\restriction\,\mtcl U^{\alpha_i}$-name for such a decomposition.

The proof of the present version of Lemma \ref{symm-lem}---with the exact same statement---is immediate using condition (5) in our present definition.
The current version of Lemma \ref{same-subsets-o1-ext} is then proved also in essentially the same way; the only difference is that now we use the present counterpart of Lemma \ref{symm-lem}.

 The current version of Lemma \ref{properness} is also proved by induction on $\alpha<\kappa$.
The proof of $(1)_\alpha$ (and of $(2)_\alpha$) is the same as in the original lemma when $\alpha$ is $0$, a successor ordinal, or an ordinal of countable cofinality. When $\cf(\alpha)\geq\omega_1$, we pick $\bar\alpha$, $G$, $r$ and $q$ as in the old proof, and we fix $(\beta_i)_{i<n}$ as we did there. We then build a decreasing sequence $(q^+_i)_{i<n}$ as in that proof. In order to argue that $q^+_i$ can be taken to satisfy the counterpart of point (3) in that proof, we argue as in the first construction, using the current version of Lemma \ref{same-subsets-o1-ext} together with the definition of symmetrically proper forcing.

The proof of the current version of Lemma \ref{properness1} is essentially the same as in the original lemma.

The proof of the corresponding versions of Corollary \ref{cor-prop} and Lemmas \ref{pfaomega_1} and \ref{arithm} is the same, using what we have already established, as for the original results.

We will need the following more general form of Lemma \ref{properness1}.

\begin{lemma}\label{properness1+}
Let $\alpha<\kappa$,  $p\in\bbP_\alpha$, and let $(N_0, \alpha)$ and $(N_1, \alpha)$ be models with marker such that $N_0$, $N_1\in\mtcl E_{\alpha+1}$ and $(N_0; \in, \Phi_{\alpha+1}\cap N_0)\cong (N_1; \in, \Phi_{\alpha+1}\cap N_1)$, $\Psi_{N_0, N_1}$ is the identity on $N_0\cap N_1$, and $p\in N_0$. 
Then there is an extension $p^*\in\bbP_\alpha$ of $p$ such that $(N_i, \rho)\in \Delta_{p^*}$ for all $i < 2$ and all $\rho\in N_i\cap (\alpha+1)$.
\end{lemma}

The proof of Lemma \ref{properness1+} is by induction on $\alpha$ and very similar to that of the present version of Lemma \ref{properness1}.

The following $\CH$-preservation lemma is proved using the present version of Lemma \ref{symm-lem}.\footnote{See for example \cite{fnr} for a similar argument.}

\begin{lemma}\label{ch-pres}
For every $\alpha<\kappa$, $\bbP_\alpha\restriction\,\mtcl U^\alpha$ forces $\CH$.
\end{lemma}

\begin{proof}
Let $\name r_i$, for $i<\omega_2$, be $\bbP_\alpha\restriction\,\mtcl U^\alpha$-names for subsets of $\omega$ and suppose, towards a contradiction, that $p\in \bbP_\alpha\restriction\,\mtcl U^\alpha$ forces $\name r_i\neq \name r_{i'}$ for all $i\neq i'$. By the $\aleph_2$-c.c.\ of $\bbP_\alpha\restriction\,\mtcl U^\alpha$ we may of course assume that these names are all in $\cH(\kappa)$. For each $i$ let $M_i^*$ be a countable elementary submodel of $\cH(\kappa^+)$ containing everything relevant, which includes $p$ and $\name r_i$. Let $\mtcl M_i$ be a structure with universe $M_i:=M_i^*\cap\cH(\kappa)$ coding $\alpha$, $p$, $\name r_i$ and $\langle (i, x)\,:\, \beta\in M_i\cap\kappa, x\in \Phi_\beta\cap M_i\rangle$ in some canonical way. By $\CH$ we may find $i_0\neq i_1$ such that $\mtcl M_{i_0}\cong\mtcl M_{i_1}$ and $\Psi_{M_{i_0}, M_{i_1}}$ is the identity on $M_{i_0}\cap M_{i_1}$.

By Lemma \ref{properness1+} we may extend $p\in\bbP_\alpha$ to a condition $p^*\in\bbP_\alpha$ such that $(M_{i_0}, \rho_0)$, $(M_{i_1}, \rho_1)\in\Delta_{p^*}$ for al $\rho_0\in M_{i_0}\cap (\alpha+1)$ and $\rho_1\in M_{i_1}\cap (\alpha+1)$. By the present form of Lemma \ref{symm-lem} we have that for every $s\in (\bbP_\alpha\restriction\,\mtcl U^\alpha)\cap M_{i_0}$ and every $p'\in\bbP_\alpha$ extending $p^*$, $p'\leq_{\bbP_\alpha\restriction\,\mtcl U^\alpha} s$ if and only if $p'\leq_{\bbP_\alpha\restriction\,\mtcl U^\alpha}\Psi_{M_{i_0}, M_{i_1}}(s)$.

Finally, by the present version of Lemma \ref{properness}, for every condition $p'\in\bbP_\alpha$ extending $p^*$ and every $n<\omega$, if $p'$ extends a condition deciding the truth value of the statement $n\in \name r_{i_0}$, then $p'$ can be extended to a condition stronger than some $s\in (\bbP_\alpha\restriction\,\mtcl U^\alpha)\cap M_{i_0}$ deciding this statement (of course in the same way). Hence, by the previous paragraph and the fact that $$\Psi_{M_{i_0}, M_{i_1}}:(M_{i_0}; \in, \Phi_{\alpha+1}, \name r_{i_0})\rightarrow (M_{i_1}; \in, \Phi_{\alpha+1}, \name r_{i_1})$$ is an isomorphism sending $\name r_{i_0}$ to $\name r_{i_1}$, we get that $p^*\Vdash_{\bbP_\alpha\restriction\,\mtcl U^\alpha}\name r_{i_0}=\Psi_{M_{i_0}, M_{i_1}}(\name r_{i_0})=\name r_{i_1}$. This is a contradiction since $p^*$ extends $p$.
\end{proof}

The following is an immediate consequence of Lemma \ref{ch-pres}, together with the $\aleph_2$-c.c.\ of $\bbP_\kappa$ and our choice of book-keeping.

\begin{corollary}\label{local-ch}
$\bbP_\kappa$ forces Local $\CH$.
\end{corollary}

\begin{lemma}\label{prbpfa}
$\bbP_\kappa$ forces $\CH$--$\omega_1$\text{-linked-Symm-}$\BPFA(\kappa)$.
\end{lemma}

\begin{proof}
Let $G$ be $\bbP_\kappa$-generic over $\bold V$, let $a\in\cH(\aleph_2)^{\bold V[G]}$, let $\varphi(x, y)$ be a $\Sigma_0$ formula, and suppose for every ground model $M$ of $\bold V[G]$ such that $a\in M$ and $M\models\CH$ it holds in $M$ that there is an $\omega_1$-linked symmetrically proper forcing $\bbQ\sub\cH(\kappa)^M$ such that $\bbQ$ forces $\cH(\aleph_2)\models \exists y\varphi(a, y)$. By Corollary \ref{local-ch}, it will be enough to show that there is some $b\in \bold V[G]$ such that $\bold V[G]\models \varphi(a, b)$.

Let $\name a$ be a nice $\bbP_\kappa$-name for a subset of $\omega_1$ coding $a$ and let $\vec A=\langle A_i\,:\,i<\omega\rangle$ be a sequence of antichains of $\bbP_\kappa$ such that   $\name a=\bigcup_{i<\omega_1}\{i\}\times A_i$.  Using the $\aleph_2$-c.c.\ of $\bbP_\kappa$ and the choice of $\phi$, we may find $\alpha$ such that $\phi(\alpha)=(\varphi(x, y), \langle A_i\,:\, i<\omega_1\rangle)$. Let $G_0=G\cap\bbP_\alpha\restriction\,\mtcl U^\alpha$ and let $M=\bold V[G_0]$. Then $a\in M$ and, by Lemma \ref{local-ch}, $M\models\CH$. By our hypothesis there is in $M$ an $\omega_1$-linked symmetrically proper forcing $\bbQ\sub \cH(\kappa)^M$ adding a witness to $\exists y\varphi(a, y)$. It is easy to see that $(2^{{<}\kappa})^M=\kappa$. Hence, in $M$ there is in fact an $\omega_1$-linked symmetrically proper forcing $\bbQ\sub \kappa$ adding a witness to $\exists y\varphi(a, y)$. This concludes the proof since then some condition in $G_0$ will force $\name{\bbQ}_\alpha$ to be such a forcing, from which it will follow that $\bold V[G]\models \exists y\varphi(y, a)$ as desired.
\end{proof}

\begin{remark}
We do not know whether $\CH$--$\omega_1$\text{-linked-Symm-}$\BPFA(\mu)$, for any given $\mu$, implies the unbounded form of this axiom, i.e., $\CH$--$\omega_1$\text{-linked-Symm-}$\BPFA(\lambda)$ for every cardinal $\lambda$. We do not even know if our model satisfies this unbounded form.
\end{remark}

Lemma \ref{prbpfa} completes the proof of Theorem \ref{thm2}.

We have the following implication.

\begin{proposition}\label{princ-impl-MA} $\CH$--$\omega_1$\text{-linked-Symm-}$\BPFA(\omega_1)$ implies $\MA_{\omega_1}$.\end{proposition}

\begin{proof}
It is enough to show that if $\bbP$ is a poset on $\omega_1$ with the c.c.c.\ and $\mtcl D$ is a collection of $\aleph_1$-many dense subsets of $\bbP$, then there is a filter of $\bbP$ meeting all members of $\mtcl D$. Let $M$ be any ground model of $\CH$ containing $\bbP$ and $\mtcl D$ and computing $\omega_1$ correctly---say, containing in addition some fixed sequence $(e_\alpha\,:\,\alpha<\omega_1)$ of bijections $e_\alpha:|\alpha|\to\alpha$. Then it holds in $M$ that $\bbP$ has the c.c.c.\ as every uncountable antichain of $\bbP$ from the point of view of $M$ would be uncountable in $\bold{V}$. But then $\bbP$ is of course  $\omega_1$-linked and symmetrically proper in $M$, and an application of $\CH$--$\omega_1$\text{-linked-Symm-}$\BPFA(\omega_1)$ to $M$ and $\bbP$ yields the existence in $\bold{V}$ of a filter of $\bbP$ meeting all members of $\mtcl D$. 
\end{proof}

On the other hand, it is unknown to us whether $\PFA$ implies our principle $\CH$--$\omega_1$\text{-linked-Symm-}$\BPFA(\omega_1)$, or any reasonable variant of it. We suspect that it does not. In fact,  $\CH$--$\omega_1$\text{-linked-Symm-}$\BPFA(\omega_1)$ and natural variants thereof seem to be orthogonal to ordinary forcing axioms stronger than $\MA_{\omega_1}$ (s.\ the comment after Question \ref{q213}).

The following is another natural question.

\begin{question}\label{q213} One can naturally define the strengthening $\CH$--$\aleph_2$\text{-c.c.-Symm-}$\BPFA(\omega_2)$ of our principle  $\CH$--$\omega_1$\text{-linked-Symm-}$\BPFA(\omega_2)$. Is this principle consistent?
\end{question}

It is also worth pointing out that our proof of Proposition \ref{princ-impl-MA} does not suffice to derive $\PFA(\omega_1)$ from $\CH$--$\omega_1$\text{-linked-Symm-}$\BPFA(\omega_1)$. Indeed, there is a priori no reason why, given a proper $\bbP\sub\omega_1$ and a ground model $M$ satisfying $\CH$ and such that $\bbP\in M$ (and possibly $a\in M$ for some other specified $a\in\cH(\omega_2)$), $\bbP$ should also be proper in $M$ (stationary sets in $M$ of the form $[X]^{\aleph_0}$ need not be stationary in $\bold{V}$). The following is thus another natural question.

\begin{question} Does $\CH$--$\omega_1$\text{-linked-Symm-}$\BPFA(\omega_1)$ impliy $\PFA(\omega_1)$?
\end{question}

\subsection{Some applications of $\CH$--$\omega_1$\text{-linked-Symm-}$\BPFA(\omega_2)$}\label{applications}

We will finish this section by showing that $\CH$--$\omega_1$\text{-linked-Symm-}$\BPFA(\omega_2)$ implies a number of interesting combinatorial consequences of $\PFA$. We will focus on Baumgartner's Axiom for $\aleph_1$-dense sets of reals, Todor\v{c}evi\'{c}'s  Open Colouring Axiom for sets of size $\aleph_1$, the Abraham-Rubin-Shelah Open Colouring Axiom, Moore's Measuring principle, Baumgartner's Thinning-Out Principle, and Todor\v{c}evi\'{c}'s P-ideal Dichotomy for $\aleph_1$-generated ideals on $\omega_1$. Hence, by Theorem \ref{thm2}, all these statements are simultaneously compatible with $2^{\aleph_0}>\aleph_2$.

A set $A$ of reals is $\aleph_1$-dense if $A\cap (x,\, y)$ has cardinality $\aleph_1$ for all reals $x<y$, where $(x,\,y)$ denotes the open interval $\{z\in\mathbb R\,:\, x<z<y\}$.
\emph{Baumgartner's axiom for $\aleph_1$-dense sets of reals}, $\BA$, is the statement that all $\aleph_1$-dense sets of reals are order-isomorphic.

By a well-known theorem of Baumgartner (\cite{baumgartner0}), $\BA$ can be forced by a proper forcing. In fact, the following holds.

\begin{lemma}\label{baumgartner} (\cite{baumgartner0}) Assume $\CH$ holds and suppose $A$ and $B$ are $\aleph_1$-dense sets of reals. Then there is a c.c.c.\ partial order $\bbQ$ of cardinality $\aleph_1$ adding an order-isomorphism $\pi:A\to B$. \end{lemma}

We will also mention that \cite{ARS} proves the consistency of $\BA$ with $2^{\aleph_0}>\aleph_2$.

Given a set $X$, $\Id_X$, the identity on $X$, is $\{(x, x)\,:\, x\in X\}$.  A colouring of a set of reals $X$ is a partition $(K_0, K_1)$ of $X\setminus \Id_X$ such that  for all distinct $x$, $y\in X$, $(x, y)\in K_0$ if and only if $(y, x)\in K_0$. We say that $(K_0, K_1)$ is open if $K_0$ is an open subset of $X\setminus \Id_X$ with the product topology.

Given a colouring $(K_0, K_1)$ of $X$ and $i\in\{0, 1\}$, we say that $H\subseteq X$ is $i$-homogeneous if $(x, y)\in K_i$ for all distinct $x$, $y\in H$.

Todor\v{c}evi\'{c}'s Open Colouring Axiom (\cite{To1}), which we will denote by $\OCA$, is the statement that if $(K_0, K_1)$ is an open colouring of a set $X$ of reals, then one of the following holds:

\begin{enumerate}
\item There is an uncountable $0$-homogeneous subset of $X$.
\item There is a sequence $(X_n)_{n<\omega}$ such that $X=\bigcup_{n<\omega} X_n$ and each $X_n$ is $1$-homogeneous.
\end{enumerate}

$\OCA$ follows from $\PFA$ (\cite{To1}). In fact we have the following.

\begin{lemma}\label{velickovic} (\cite{velickovic1})
Let $X$ be a set of reals and suppose $(K_0, K_1)$ is an open colouring of $X$. Suppose $K_0$ is not a union of ${<}2^{\aleph_0}$-many $1$-homogeneous sets. Then there is $Y\subseteq X$ of size $2^{\aleph_0}$ such that the poset of finite $0$-homogeneous subsets of $Y$, ordered by reverse inclusion, has the $2^{\aleph_0}$-chain condition.
\end{lemma}

$\OCA(\aleph_1)$ will denote the restriction of $\OCA$ to colourings of sets $X\subseteq\mathbb R$ of cardinality $\aleph_1$. $\OCA(\aleph_1)$ is a useful fragment of $\OCA$; for example, it is easy to see that $\OCA(\aleph_1)$ implies that every uncountable subset of $\mathcal P(\omega)$ contains an uncountable chain or an uncountable antichain under inclusion (s.\ \cite{velickovic1}). Farah proved that $\OCA(\aleph_1)$ is consistent together with $2^{\aleph_0}$ large.

The Abraham-Rubin-Shelah Open Colouring Axiom, introduced  in \cite{ARS}, which we will denote by $\OCA_{[\text{ARS}]}$, is the statement that if $X$ is a second countable Hausdorff topological space of size $\aleph_1$ and $(K_0, K_1)$ is an open colouring of $X$, then there exists a partition $X=\bigcup_{n<\omega}X_n$ of $X$ such that each $X_n$ is homogeneous. In \cite{ARS}, it is shown that $\OCA_{[\text{ARS}]}$ implies the failure of $\CH$ and is consistent with $2^{\aleph_0}=\aleph_2$. The question whether $\OCA_{[\text{ARS}]}$ is consistent with
large values of the continuum was asked in that paper.

Given a space $X$ as above, an open partition $(K_0, K_1)$ of $X$, and a function $f: X \rightarrow \{0, 1\}$, let $\bbQ(K_0, K_1, f)$
be the following forcing notion. A condition $p$ is in $\bbQ(K_0, K_1, f)$ if and only if
\begin{enumerate}
\item $p$ is a finite partial function from $\omega$ into $[X]^{<\omega}$,
\item for all $n \in \dom(p),$ $f \restriction p(n)$ is constant, say with value $i$, and
\item $p(n)$ is $i$-homogeneous.
\end{enumerate}
The ordering is the natural one: $q \leq_{\bbQ(K_0, K_1, f)} p$ iff
\begin{enumerate}
\item $\dom(q) \supseteq \dom(p)$;
\item for all $n \in \dom(p),$ $q(n) \supseteq p(n)$.
\end{enumerate}

It is clear that any forcing of the form $\bbQ(K_0, K_1, f)$ adds a partition $X=\bigcup_{n<\omega}X_n$ of $X$ into homogeneous pieces for $(K_0, K_1)$.

\begin{lemma}
\label{abraham}
(\cite{ARS}) Assume $\CH$. Suppose $X$ is a second countable Hausdorff topological space of size $\aleph_1$ and $(K_0, K_1)$ is an open colouring of $X$. Then,  for some function $f:X \rightarrow \{0, 1\}$, the forcing $\bbQ(K_0, K_1, f)$ is c.c.c.
\end{lemma}

We note that the consistency of $\OCA_{[\text{ARS}]}$ with $2^{\aleph_0}=\aleph_3$ was recently proved by Gilton-Neeman in \cite{gilton2}, but the consistency of this statement with $2^{\aleph_0}>\aleph_3$ remained open. It follows by our results, together with Lemma \ref{abraham}, that $\OCA_{[\text{ARS}]}$ is compatible with arbitrarily large continuum, which answers the Gilton-Neeman question.

In \cite{moore}, Moore proves that the conjunction of Todor\v{c}evi\'{c}'s $\OCA$ and $\OCA_{[\text{ARS}]}$ implies $2^{\aleph_0}=\aleph_2$. It follows that in our model for Theorem \ref{thm2} for $\kappa>\aleph_2$, $\OCA$ fails and hence there is an open colouring $(K_0, K_1)$ of a set $X$ of reals with the following incompactness property: there is no uncountable $0$-homogeneous subset of $X$, every $Y\in [X]^{\aleph_1}$ is a union of countably many $1$-homogeneous sets, but $X$ itself is not.

Measuring, defined by Moore (s.\ \cite{EMM}), is the statement that for every club-sequence $\vec C=\langle C_\delta\,:\,\delta\in \omega_1\rangle$\footnote{I.e., each $C_\delta$ is a club of $\delta$.} there is a club $C\subseteq\omega_1$ with the property that for every $\delta\in C$ there is some $\alpha<\delta$ such that either

\begin{itemize}
\item $(C \cap\delta)\setminus\alpha\subseteq C_\delta$, or
\item $(C\setminus\alpha)\cap C_\delta=\emptyset$.
\end{itemize}

In the above situation, we say that $C$ measures $\vec C$.

Measuring follows from $\PFA$ and can be regarded as a strong failure of Club Guessing at $\omega_1$; in fact it is easy to see that it implies $\lnot\WCG$.

Given a club-sequence $\vec C=\langle C_\delta\,:\,\delta\in \omega_1\rangle$, there is a natural proper forcing $\bbQ_{\vec C}$ for adding a club of $\omega_1$ measuring $\vec C$  (s.\ \cite{EMM}). Conditions in $\bbQ_{\vec C}$ are pairs $(x, C)$ such that:

\begin{enumerate}
\item $x$ is a closed countable subset of $\omega_1$;
\item for every $\delta\in\Lim(\omega_1)\cap x$ there is some $\alpha<\delta$ such that either
\begin{itemize}
\item $(x \cap\delta)\setminus\alpha\subseteq C_\delta$, or
\item $(x\setminus\alpha)\cap C_\delta=\emptyset$.
\end{itemize}
\item $C$ is a club of $\omega_1$.
\end{enumerate}

Given $\bbQ_{\vec C}$-conditions $(x_0, C_0)$, $(x_1, C_1)$, we let $(x_1, C_1)\leq_{\bbQ_{\vec C}}(x_0, C_0)$ iff
\begin{enumerate}
\item $x_1$ is an end-extension of $x_0$ (i.e., $x_0\subseteq x_1$ and $x_1\cap (\max(x_0)+1)=x_0$),
\item $C_1\subseteq C_0$, and
\item $x_1\setminus x_0\subseteq C_0$.
\end{enumerate}

\begin{lemma}\label{measuring} $\bbQ_{\vec C}\sub\cH(\omega_2)$ is a symmetrically proper P\v{r}\'{i}kr\'{y}-type forcing notion with stems in $\cH(\omega_1)$. \end{lemma}

\begin{proof}
The fact that $\bbQ_{\vec C}$ is a  P\v{r}\'{i}kr\'{y}-type forcing notion with stems in $\cH(\aleph_1)$ is clear from the definition, and it is obvious that $\bbQ_{\vec C}\sub\cH(\omega_2)$. The properness of $\bbQ_{\vec C}$ is a standard fact (s.\ e.g.\ \cite{EMM}). We will now show, by a variation of the proof in \cite{EMM}, that $\bbQ_{\vec C}$ is actually symmetrically proper.

For this, suppose $\mtcl N$ is a finite symmetric system of countable elementary submodels of  $(\cH(\omega_2); \in, \vec C)$ and let $(x, C)$ be a condition in $\bbQ_{\vec C}\cap N$ for some $N\in\mtcl N$ of minimal height within $\mtcl N$. We will find a condition $(x^*, C^*)\in\bbQ_{\vec C}$ which is $\mtcl N$-symmetric and $(N',\bbQ_{\vec C})$-generic for all $N'\in \mtcl N$. For this, let $(N_k)_{k\leq n}$, for some $n<\omega$, be a $\subseteq$-maximal $\in$-chain of members of $\mtcl N$. Let also $N_{n+1}=\bold V$. We build a $\bbQ_{\vec C}$-decreasing sequence $(x_k, C_k)$ of $\bbQ_{\vec C}$, for $k\leq n+1$, letting $(x_0, C_0)=(x, C)$ and making sure that $(x_{k+1}, C_{k+1})\in N_{k+1}$ is $(N_k, \bbQ_{\vec C})$-generic.

In order to find $(x_{k+1}, C_{k+1})$ we build, working in $N_{k+1}$, a $(N_k, \bbQ_{\vec C})$-generic sequence $((x^i_k, C^i_k)\,:\,i<\omega)$ of conditions extending $(x_k, C_k)$ and belonging to $N_k$; i.e., $(x^{i+1}_k, C^{i+1}_k)\in N_k$ extends $(c^i_k, C^i_k)$ for all $i$ and for every maximal antichain $A\in N_k$ of $\bbQ_{\vec C}$ there is some $i$ with $(x^i_k, C^i_k)$ extending some condition in $A$. Further, if there is a club $C\in N_k$ of $\omega_1$ such that $C\cap\delta_{N_k}\sub C_{\delta_{N_k}}$, then we pick $C^0_k$ so that $C\sub C^0_k$ for some such $C$; and if, on the other hand, $C\setminus C_{\delta_{N_k}}\neq\emptyset$ for every club $C\in N_k$ of $\omega_1$, then we make sure that $(\bigcup_{i<\omega}x^i_k\setminus \max(x_k))\cap C_{\delta_{N_k}}=\emptyset$ (this is a standard $\MRP$-type construction; s.\ \cite{EMM}). In either case we let $$(x_{k+1}, C_{k+1})=(\bigcup_{i<\omega}x^i_k\cup\{\delta_{N_k}\}, \bigcap (\Club_{\omega_1}\cap \bigcup (\mtcl N\cap N_{k+1}))),$$ where $\Club_{\omega_1}$ denotes the club filter on $\omega_1$, and note that $(x_{k+1}, C_{k+1})\leq_{\bbQ_{\vec C}}(x^i_k, C^i_k)$ for each $i$.

Let $(x^*, C^*)=(x_{n+1}, \bigcap (\Club_{\omega_1}\cap\,\mtcl N))$, 
and let us note that $(x^*,C^*)\leq_{\bbQ_{\vec C}}(x_k, C_k)$ for each $k$, and that in fact $(x^*, C^*)\leq_{\bbQ_{\vec C}}(x^i_k, \Psi_{N_k, N}(C^i_k))$ for all $k$, $i$ and all $N\in\mtcl N$ with $\delta_N=\delta_{N_k}$.  Given any $N\in\mtcl N$, we check that $(x^*, C^*)$ is $(N, \bbQ_{\vec C})$-generic. For this, let $(x', C')\leq_{\bbQ_{\vec C}}(x^*, C^*)$ and $A\in N$ be a maximal antichain of $\bbQ_{\vec C}$ and let us check that $(x', C')$ is compatible with a condition in $E\cap N$. Let $k$ be such that $\delta_N=\delta_{N_k}$ and let $B=\Psi_{N, N_k}(A)$. Then there is some $i$ such that $(x^i_k, C^i_k)$ extends some condition $s\in B\cap N_k$. But then $(x', C')$ extends $(x^i_k, \Psi_{N_k, N}(C^i_k))$, which in turn extends $\Psi_{N_k, N}(s)\in A\cap N$.

Finally, we show that $(x^*, C^*)$ is $\mtcl N$-symmetric. For this, let again $(x', C')\in\bbQ_{\vec C}$ be an extension of $(x^*, C^*)$, let $N^0$, $N^1\in\mtcl N$ be such that $\delta_{N^0}=\delta_{N^1}$, let $(x, C)\in N^0\cap\bbQ_{\vec C}$ be such that $(x', C')\leq_{\bbQ_{\vec C}}(x,C)$, and let us check that $(x', C')\leq_{\bbQ_{\vec C}}\Psi_{N^0, N^1}((x, C))=(x, \Psi_{N^0, N^1}(C))$.  Since of course $C'\supseteq\Psi_{N^0, N^1}(C)$, in order to see that  $(x', C')\leq_{\bbQ_{\vec C}}(x, \Psi_{N^0, N^1}(C))$ it suffices to show that $x'\setminus x\sub\Psi_{N^0, N^1}(C)$. Let $k\leq n$ be such that $\delta_{N^0}=\delta_{N^1}=\delta_{N_k}$. Since $(x'\cap\delta_{N_k})\setminus x\sub C$ and $\Psi_{N^0, N^1}(C)\cap\delta_{N_k}=C\cap\delta_{N_k}$, we have that $(x'\cap\delta_{N_k})\setminus x\sub\Psi_{N^0, N^1}(C)$. Also, for every $j$ such that $k<j\leq n$, if $N\in\mtcl N$ is such that $\delta_N=\delta_{N_j}$ and $N^1\in N$, then $(x'\cap \delta_{N_j})\setminus\delta_{N_k}\sub \Psi_{N, N_j}(\Psi_{N^0, N^1}(C))$ by the construction of $(x^*, C^*)$, but of course $\Psi_{N, N_j}(\Psi_{N^0, N^1}(C))\cap\delta_{N_j}=\Psi_{N^0, N^1}(C)\cap\delta_{N_j}$.  And finally, $x'\setminus \delta_{N_n}\sub \Psi_{N^0, N^1}(C)$ again by the construction of $(x^*, C^*)$.
\end{proof}

The Thinning-Out Principle ($\TOP$) is the following statement, defined by Baumgartner in \cite{baumgartner}: Suppose $A\subseteq\omega_1$, $B\subseteq\omega_1$, and $(B_\alpha\,:\,\alpha\in B)$ is such that $B_\alpha\subseteq \alpha$ for each $\alpha$. Suppose for every uncountable $X\subseteq A$ there is some $\beta<\omega_1$ such that $$\{X\}\cup\{B_\alpha\,:\,\alpha\in B\setminus\beta\}$$ has the finite intersection property. Then there is an uncountable $X\subseteq A$ such that $(X\cap \alpha)\setminus B_\alpha$ is finite for every $\alpha\in B$.

The conjunction $MA_{\aleph_1}+\TOP$ has several interesting consequences; for example, it implies the non-existence of S-spaces, the partition relation $\omega_1\rightarrow (\omega_1, \alpha)^2$ for each $\alpha<\omega_1$, and that if $(D, \leq_D)$ is a directed set of size $\aleph_1$ and every uncountable subset of $D$ contains a countable unbounded set, then there is an uncountable subset $X$ of $D$ such that every infinite subset of $X$ is unbounded (s.\ \cite{baumgartner}).

\begin{lemma}\label{baumgartner1}
Suppose $A$ and $\vec B=(B_\alpha\,:\,\alpha \in B)$ are as in the statement of $\TOP$. Then there is a forcing notion $\bbQ_{A, \vec B}$ adding $X$  as in the conclusion of $\TOP$ for $A$ and $\vec B$ and such that $\bbQ_{A, \vec B}$ is of the form $\bbQ_0\ast\name\bbQ_1\ast\name\bbQ_2$, where
\begin{enumerate}
\item $\bbQ_0$ is $\Add(\omega, \omega_1)$, i.e., the standard forcing for adding $\aleph_1$-many Cohen reals,
\item $\name\bbQ_1$ is an $\Add(\omega, \omega_1)$-name for the standard forcing for adding a club diagonalizing the club filter on $\omega_1$ in $\bold V[\name{G}_{\bbQ_0}]$, and
\item $\name\bbQ_2$ is a $\bbQ_0\ast\name\bbQ_1$-name for a c.c.c.\ forcing of size $\aleph_1$.
\end{enumerate}

In particular, if $\CH$ holds, then $\bbQ_{A, \vec B}$  is an $\omega_1$-linked symmetrically proper forcing notion included in $\cH(\omega_2)$.
\end{lemma}

\begin{proof}
The first assertion of the lemma is proved in \cite{baumgartner}. As to the second assertion, we first observe that the standard forcing $\bbQ$ for adding a club diagonalizing the club filter on $\omega_1$ is a symmetrically proper P\v{r}\'{i}kr\'{y}-type forcing notions with stems in $\cH(\aleph_1)$---and therefore it is $\omega_1$-linked if $\CH$ holds. $\bbQ$ is the partial order of pairs $(x, C)$, where
\begin{itemize}
\item $x$ is a closed countable subset of $\omega_1$ and
\item $C$ is a club of $\omega_1$,
\end{itemize}
\noindent and where $(x_1, C_1)$ extends $(x_0, C_0)$ if
\begin{itemize}
\item $x_1$ is an end-extension of $x_0$,
\item $C_1\subseteq C_0$, and
\item $x_1\setminus x_0\subseteq C_0$.\end{itemize}

This is trivially a P\v{r}\'{i}kr\'{y}-type forcing notion with stems in $\cH(\aleph_1)$, and the fact that it is symmetrically proper is established by a (simpler) version of the proof of Lemma \ref{measuring}. But now, if $\CH$ holds, then we have that $\bbQ_{A, \vec B}$, being an iteration of three $\omega_1$-linked symmetrically proper forcings, has itself this property by Proposition \ref{two-step}.\footnote{One can prove, using the fact that every club of $\omega_1$ in the second generic extension contains a club of $\omega_1$ in $\bold V$, that  $\bbQ_{A, \vec B}$ is in fact a P\v{r}\'{i}kr\'{y}-type forcing notion with stems in $\cH(\aleph_1)$.} Finally, it is clear that two-step iterations of $2^{\aleph_1}$-sized forcings with the $\aleph_2$-c.c.\ are themselves $2^{\aleph_1}$-sized---where we are taking all our names for subsets of $\omega_1$ to be nice names---and hence $|\bbQ_{A, \vec B}|=2^{\aleph_1}$ if $\CH$ holds.
\end{proof}

The last combinatorial principle we will consider in this subsection concerns ideals $\mathcal I\subseteq [S]^{{\leq}\aleph_0}$ on some set $S$ consisting of countable sets and containing all finite subsets of $S$. Such an ideal $\mathcal I$ is said to be a P-ideal in case for every sequence $(X_n)_{n<\omega}$ of members of $\mathcal I$ there is some $Y\in \mathcal I$ such that $X_n\setminus Y$ is finite for each $n$.

Todor\v{c}evi\'{c}'s \emph{P-ideal Dichotomy} is the statement that for every set $S$ and every P-ideal $\mathcal I\subseteq [S]^{{\leq}\aleph_0}$ on $S$, either
\begin{enumerate}
\item there is an uncountable $X\subseteq S$ such that $[X]^{\aleph_0}\subseteq \mathcal I$, or
\item $S=\bigcup_{n\in\omega}X_n$ for some sequence $(X_n)_{n\in\omega}$ such that $X_n\cap I$ is finite for every $n<\omega$ and every $I\in\mathcal I$.
\end{enumerate}

Given an ideal $\mathcal I$, $\mathcal J\subseteq\mathcal I$ is a \emph{generating set of $\mathcal I$} if $\mathcal I=\{X\,:\,X\subseteq Y\mbox{ for some }Y\in\mathcal J\}$. Also, we say that \emph{$\mathcal J$ generates $\mathcal I$}. An ideal $\mathcal I$ on $\omega_1$ is said to be $\aleph_1$-generated if there is $\mathcal J$, a generating set of $\mathcal I$, such that $|\mathcal J|=\aleph_1$.

Given sets $X$, $Y$, $X\sub^* Y$ means that $X\setminus Y$ is finite. An \emph{$\omega_1$-tower} is a sequence $\langle X_\alpha\,:\, \alpha<\omega\rangle$ of countable subsets of $\omega_1$ such that $X_\alpha\sub^\ast X_\beta$ for all $\alpha<\beta$.
It is clear that every $\aleph_1$-generated $P$-ideal is in fact generated by an $\omega_1$-tower.

We have the following.

\begin{lemma}\label{todor-abr}
Suppose $\CH$ holds, $\mtcl I\sub [\omega_1]^{{\leq}\aleph_0}$ is a $P$-ideal, and $\vec X=\langle X_\alpha\,:\,\alpha<\omega_1\rangle$ is an $\omega_1$-tower generating $\mtcl I$. Suppose $\omega_1$ cannot be decomposed into countably many sets $X$ such that $X\cap I$ is finite for each $I\in \mtcl I$. Let $\bbQ_{\vec X}$ be the poset consisting of pairs $p=(x_p, A_p)$ such that
\begin{itemize}
\item $x_p\in [\omega_1]^{{\leq}\aleph_0}$ and
\item $A_p=[\omega_1]^{\aleph_1}\setminus B_p$ for some $B_p\in [[\omega_1]^{\aleph_1}]^{{\leq}\aleph_0}$,
\end{itemize}

\noindent where $(x_q, A_q)\leq_{\bbQ_{\vec X}}(x_p, A_p)$ if and only if
\begin{itemize}
\item $x_q$ is an end-extension of $x_p$ (i.e., $x_p\sub x_q$ and $x_q\cap\sup\{\alpha+1\,:\,\alpha\in x_p\}=x_p$),
\item $A_q\sub A_p$, and
\item for every $X\in B_p$, $\{\xi\in X\,:\, x_q\setminus x_p\sub X_\xi\}$ is uncountable and belongs to $B_q$.
\end{itemize}

Then

\begin{enumerate}
\item $\bbQ_{\vec X}\sub\cH(\omega_2)$,
\item $\bbQ_{\vec X}$ forces the existence of some $X\in [\omega_1]^{\aleph_1}$ such that $[X]^{\aleph_0}\sub \mtcl I$,
\item $\bbQ_{\vec X}$ is symmetrically proper, and
\item $\bbQ_{\vec X}$ is a P\v{r}\'{i}kr\'{y}-type forcing with stems in $\cH(\aleph_1)$.
\end{enumerate}
\end{lemma}

\begin{proof}
(1) is obvious.
(2) and the properness of $\bbQ_{\vec X}$ are proved in \cite{Abr-To}---albeit with the (complementary) presentation of the forcing given by $(x_p, B_p)$ rather than $(x_p, A_p)$. (3) is immediate by the presentation of $\bbQ_{\vec X}$. Finally, the symmetric properness of $\bbQ_{\vec X}$ follows by a construction very similar to the one in the proof of Lemma \ref{measuring}.
\end{proof}

The following corollary is now a consequence from Lemmas \ref{baumgartner}, \ref{velickovic}, \ref{abraham}, \ref{measuring}, \ref{baumgartner1}, and \ref{todor-abr}.

\begin{corollary}\label{final-cor}
The following statements follow from $\CH$--$\omega_1$\text{-linked-Symm-}$\BPFA(\omega_2)$.

\begin{enumerate}
\item $\BA$,
\item $\OCA(\aleph_1)$,
\item $\OCA_{[\text{ARS}]}$,
\item Measuring,
\item $\TOP$,
\item The P-ideal Dichotomy for $\aleph_1$-generated ideals on $\omega_1$.
\end{enumerate}
\end{corollary}

\section{$\MM(\omega_1)$ with large continuum}\label{section3}

Let $\MM(\omega_1)$ be the restriction of Martin's Maximum to posets of size $\aleph_1$, i.e., the forcing axiom $\FA(\{\bbQ\,:\,\bbQ\text{ preserves stationary subsets of $\omega_1$}\})_{\aleph_1}$. In \cite{foreman}, Foreman and Larson showed that the  restriction of Martin's Maximum to posets of size $\aleph_2$ implies that the continuum is $\aleph_2$.  In \cite{DKMMZ}, the authors force, over any model of $\ZFC$, so as to produce a model of $\MM(\omega_1)$ with $2^{\aleph_0}=\aleph_2$. In both \cite{DKMMZ} and \cite{foreman},  the authors  asked  whether $\MM(\omega_1)$ is compatible with $2^{\aleph_0}>\aleph_2$.

As we will next show, a small variant of the construction for Theorem \ref{thm2} produces a model which, in addition to the conclusions from that theorem, satisfies also $\MM(\omega_1)$. The theorem is the following.

\begin{theorem}\label{thm3}
Assume $\GCH$. Let $\kappa\geq\aleph_2$ be a regular cardinal. Then there is an $\aleph_2$-Knaster proper partial order $\bbP$ forcing the following statements.
\begin{enumerate}
\item $2^{\aleph_0}=2^{\aleph_1}=\kappa$
\item $\CH$--$\omega_1$\text{-linked-Symm-}$\BPFA(\omega_2)$
\item $\MM(\omega_1)_{{<}\kappa}$
\end{enumerate}
\end{theorem}

Here, $\MM(\omega_1)_{{<}\kappa}$ is of course  $\FA(\{\bbQ\,:\,\bbQ\text{ preserves stationary subsets of $\omega_1$}\})_\lambda$ for all $\lambda<\kappa$.

Most of this section is devoted to proving Theorem \ref{thm3}. We start out by presenting a notion introduced in \cite{DKMMZ}.

\begin{definition} (\cite{DKMMZ}) In an $\omega_1$-preserving forcing extension $\bold V[G]$, a \emph{continuous $\bold{V}$-reflection sequence} is a sequence $\vec M=\langle\overline{M}_\alpha\,:\,\alpha\in C\rangle$ such that:
\begin{enumerate}
\item $C$ is a club of $\omega_1$;
\item for each $\alpha\in C$, $\overline{M}_\alpha$ is the transitive collapse of some (not necessarily unique) elementary submodel $M_\alpha$ of $\cH(\omega_2)^{\bold V}$ such that $\alpha=\delta_{M_\alpha}$;
\item (\emph{continuity}) for every $\alpha\in C$ and every function $x:\alpha^{{<}\omega}\to\alpha$ in $\overline{M}_\alpha$ there is some $\gamma<\alpha$ such that $x\restriction\delta^{{<}\omega}\in\overline{M}_\delta$ for every $\delta\in C$ between $\gamma$ and $\alpha$;
\item (\emph{reflection}) for every stationary set $S\sub [\cH(\omega_2)^{\bold V}]^{\aleph_0}$ in $\bold V$, $$T^{\vec M}_S=\{\alpha\in C\,:\,\overline{M}_\alpha\text{ is the transitive collapse of some member of }S\}$$ is a stationary subset of $\omega_1$.
\end{enumerate}
\end{definition}

The following proposition is then proved in \cite{DKMMZ}.

\begin{proposition}\label{main-prop}
Let $\bold V[G]$ be an $\omega_1$-preserving forcing extension in which there is a continuous $\bold V$-reflection sequence $\vec M=\langle\overline{M}_\alpha\,:\,\alpha\in C\rangle$. In $\bold V$, let $\bbP$ be a forcing on $\omega_1$ and suppose forcing with $\bbP$ preserves this cardinal. Suppose $\nu_0\in\omega_1$, $S\in\bold V$ is a stationary subset of $[\cH(\omega_2)^{\bold V}]^{\aleph_0}$, and for every $N\in S$ there is no extension of $\nu_0$ in $\bbP$ which is $(N, \bbP)$-generic. Let $H$ be $\bbP$-generic over $\bold V[G]$ such that $\nu_0\in H$ and let $E$ be the set of $\alpha\in C$ such that $H\cap A\neq\emptyset$ for every maximal antichain $A$ of $\bbP$ such that $A\cap\alpha\in\overline{M}_\alpha$. Then $E$ is a club of $\omega_1$ and $E\cap T^{\vec M}_S=\emptyset$.
\end{proposition}

Further, in \cite{DKMMZ} it is proved that there is a forcing notion $\bbQ^*$ of cardinality $2^{\aleph_1}$ and with the $\aleph_2$-p.i.c.\ adding a continuous $\bold V$-reflection sequence. Conditions in $\bbQ^*$ are pairs $q=(a_q, b_q)$, where:

\begin{enumerate}
\item $a_q$ is a function whose domain is a closed countable subset of $\omega_1$;
\item for every $\alpha\in\dom(a_q)$, $\overline{M}^q_\alpha:=a_q(\alpha)$ is the transitive collapse of a countable elementary submodel of $\cH(\omega_2)$ such that $\delta_{\overline{M}_\alpha}=\alpha$;
\item for every $\alpha\in\dom(a_q)$ and every function $x:\alpha^{{<}\omega}\to \alpha$ in $\overline{M}^q_\alpha$ there is some $\gamma<\alpha$ such that $x\restriction\delta^{{<}\omega}\in\overline{M}^q_\delta$ for every $\delta\in\dom(a_q)$ between $\gamma$ and $\alpha$;
\item $b_q$ is a countable set of functions from $\omega_1^{{<}\omega}$ to $\omega_1$.
\end{enumerate}

Given conditions $q_0$ and $q_1$ in $\bbQ^*$, $q_1$ extends $q_0$ if and only if
\begin{itemize}
\item $\dom(a_{q_1})$ is an end-extension of $\dom(a_{q_0})$,
\item $a_{q_0}\sub a_{q_1}$ and $b_{q_0}\sub b_{q_1}$, and
\item for every $\alpha\in\dom(a_{q_1})\setminus\dom(a_{q_0})$ and every $x\in b_{q_0}$, $x\restriction\alpha^{{<}\omega}\in\overline{M}^{q_1}_\alpha$.
\end{itemize}

\begin{lemma}\label{CH2}
If $\CH$ holds, then $\bbQ^*$ is $\omega_1$-linked and symmetrically proper and it is included in $\cH(\omega_2)$.
\end{lemma}

\begin{proof}
Let $\mtcl F$ be the set of functions from $\omega_1^{{<}\omega}$ to $\omega_1$ and let  $\bbQ^*_0=\{(a, b)\,:\,(a, \mtcl F\setminus b)\in\bbQ^*\}$ ordered by setting $(a_1, b_1)\leq_{\bbQ^*_0}(a_0, b_0)$ iff $(a_1, \mtcl F\setminus b_1)\leq_{\bbQ^*}(a_0, \mtcl F\setminus b_0)$.  It is clear that $\bbQ^*_0$ is  a P\v{r}\'{i}kr\'{y}-type partial order with stems in $\cH(\omega_1)$ and that $\bbQ^*$ is isomorphic to $\bbQ^*_0$ as witnessed by the function sending $q\in\bbQ^*$ to $(a_q, \mtcl F\setminus b_q)$. Hence, $\bbQ^*_0$ is $\omega_1$-linked if $\CH$ holds. The symmetric properness of $\bbQ^*$ can be shown using essentially the same argument as for the symmetric properness of $\bbQ_{\vec C}$, for a given club-sequence $\vec C$, in the proof of Lemma \ref{measuring}.\footnote{This argument appears also in the proof of Corollary 3.10 from \cite{DKMMZ}.} Finally, we trivially have that $\bbQ^*\sub\cH(\omega_2)$.
\end{proof}

Our forcing $\bbP$ witnessing Theorem \ref{thm3} is $\bbP_\kappa$ for a construction $\langle\bbP_\alpha\,:\,\alpha\leq\kappa\rangle$, based on a certain sequence $\langle\name{\bbQ}_\alpha\,:\,\alpha<\kappa\rangle$ of names, exactly as the one for Theorem \ref{thm2} except for the fact that now we make sure that for every name $\name{\bbQ}\in\cH(\kappa)$ for a forcing notion on $\omega_1$ there are unboundedly many stages $\alpha$ in $\kappa$ at which $\name{\bbQ}_\alpha$ is forced, in $\bbP_\alpha\restriction\,\mtcl U^\alpha$, to be the following forcing.
\begin{itemize}
\item $\name{\bbQ}_\alpha=\name\bbQ$ if $\name\bbQ$ is $(\name{G}_{\bbP_\alpha\restriction\,\mtcl U^\alpha}, \bbP_\alpha)$-proper.
\item $\name{\bbQ}_\alpha=\name\bbQ^*$ if $\bbQ$ is not $(\name{G}_{\bbP_\alpha\restriction\,\mtcl U^\alpha}, \bbP_\alpha)$-proper.
\end{itemize}

We refer to the above situation by saying that our construction picks $\name{\bbQ}$ at stage $\alpha$.

It should be clear that all lemmas building up to the proof of Theorem \ref{thm2} are immune to our modification. Hence, all conclusions of Theorem \ref{thm2} hold for our present construction.


\begin{lemma}\label{lemma43} Let $\alpha<\kappa$ and suppose $\name{\bbQ}$ is a $\bbP_\alpha\restriction\,\mtcl U^\alpha$-name for a partial order on $\omega_1$. Then the following are equivalent.
\begin{enumerate}
\item $\name{\bbQ}$ is forced to be $(\name{G}_{\bbP_\alpha\restriction\,\mtcl U^\alpha}, \bbP_\alpha)$-proper.
\item $\bbP_\alpha\restriction\,\mtcl U^\alpha$ forces that there is a club $E$ of $[\cH(\kappa)^{\bold V}]^{\aleph_0}$, in $\bold V[\name{G}_{\bbP_\alpha\restriction\,\mtcl U^\alpha}]$, with the property that for all $N\in E$, if $N\in\bold V$ and there is some $q\in \bbP_\alpha$ such that $q\restriction\mtcl U^\alpha\in \name{G}_{\bbP_\alpha\restriction\,\mtcl U^\alpha}$ and $(N, \rho)\in\Delta_q$
for all $\rho\in N\cap\mtcl U^\alpha$, then for every $\nu\in \omega_1^{\bold V}\cap N[\name{G}_{\bbP_\alpha\restriction\,\mtcl U^\alpha}]$ there is some $(N[\name{G}_{\bbP_\alpha\restriction\,\mtcl U^\alpha}], \name{\bbQ})$-generic condition $\nu^*\in\omega_1^{\bold V}$ such that $\nu^*\leq_{\name{\bbQ}} \nu$.
\end{enumerate}
\end{lemma}

\begin{proof}
The implication from (1) to (2) is immediate since every function $F:[\cH(\kappa)^{\bold V}]^{{<}\omega}\to\cH(\kappa)^{\bold V}$ in $\bold V$ generates a club of $[\cH(\kappa)^{\bold V}]^{\aleph_0}$ in the $\bbP_\alpha\restriction\,\mtcl U^\alpha$-extension. For the other implication it is enough to notice that if $\name{F}$ is a $\bbP_\alpha\restriction\,\mtcl U^\alpha$-name for a function from $[\cH(\kappa)^{\bold V}]^{{<}\omega}$ to $\cH(\kappa)^{\bold V}$ generating a club of $[\cH(\kappa)]^{\bold V}$ in the extension as in (2), then any club in $\bold V$ of traces with $\cH(\kappa)$ or countable elementary submodels of any large enough $\cH(\theta)$ containing $\name{F}$ together with all other relevant objects witnesses that $\name{\bbQ}$ is forced to be $(\name{G}_{\bbP_\alpha\restriction\,\mtcl U^\alpha}, \bbP_\alpha)$-proper.
\end{proof}

The following lemma is the missing piece in the proof of Theorem \ref{thm3}.

\begin{lemma}\label{finallemmathm3}
$\bbP_\kappa$ forces $\MM(\omega_1)_{{<}\kappa}$.
\end{lemma}

\begin{proof}
Let $\name{\bbQ}$ be a $\bbP_\kappa$-name for a forcing notion on $\omega_1$ preserving stationary subsets of $\omega_1$ and let $\{\name{D}_i\,:\,i<\lambda\}$, for some $\lambda<\kappa$, be a set of names for dense subsets of $\name{\bbQ}$. By the $\aleph_2$-c.c.\ of $\bbP_\kappa$ we may of course assume that all these names are in $\cH(\kappa)$. It suffices to show that there is a stage $\alpha$ such that all these names are in fact $\bbP_\alpha$-names, our construction picks $\name\bbQ$ at $\alpha$, and such that $\name\bbQ$ is forced in $\name\bbP_\alpha\restriction\,\mtcl U^\alpha$ to be $(\name{G}_{\bbP_\alpha\restriction\,\mtcl U^\alpha}, \bbP_\alpha)$-proper -- as then $\name{\bbQ}_\alpha$ is forced to be $\name\bbQ$ and so the generic for $\bbP_\kappa$ induces a generic for $\name\bbQ$ meeting all $\name{D}_i$.

By our construction we know that there is indeed a high enough $\alpha$ at which we pick $\name\bbQ$, so we only need to show that $\name\bbQ$ is forced in $\name\bbP_\alpha\restriction\,\mtcl U^\alpha$ to be $(\name{G}_{\bbP_\alpha\restriction\,\mtcl U^\alpha}, \bbP_\alpha)$-proper. Let us assume, towards a contradiction, that this is not the case. We then have that there is a $\bbP_\kappa$-generic $G$ for which, letting $G_0=G\cap \bbP_\alpha\restriction\,\mtcl U^\alpha$, $\bbQ=\name\bbQ_{G_0}$ is not  $(G_0, \bbP_\alpha)$-proper in $\bold V[G_0]$.  Then $\bbQ_\alpha=(\name{\bbQ}_\alpha)_{G_0}$ is $\bbQ^*$ as calculated in the extension $\bold V[G_0]$. We may assume that forcing with $\bbQ$ over $\bold V[G_0]$ preserves $\omega_1$ as otherwise we get an immediate contradiction. Let $p^*\in G_0$ force all of the above in $\bbP_\alpha\restriction\,\mtcl U^\alpha$. By Lemma \ref{lemma43} there is some $\nu_0\in\omega_1$ and some stationary subset $S$ of $[\cH(\omega_2)^{\bold V}]^{\aleph_0}$ in $\bold V[G_0]$ with the following properties.
\begin{enumerate}
\item  For every $N\in S$ there is some $q\in \bbP_\alpha$ such that $q\restriction\mtcl U^\alpha\in G_0$ extends $p^*$ and such that $(N, \rho)\in\Delta_q$
for all $\rho\in N\cap \mtcl U^\alpha$.
\item For every $N\in S$ there is no extension of $\nu_0$ in $\bbQ$ which is $(N, \bbQ)$-generic.
\end{enumerate}

Given (2), we have by Proposition \ref{main-prop} that if $G(\alpha)$ is the generic for $(\name{\bbQ}_\alpha)_{G_0}$ induced by $G$, $\vec M=\langle M_\alpha\,:\,\alpha\in C\rangle$ is the corresponding continuous $\bold V[G_0]$-reflection sequence obtained from $G(\alpha)$, $H$ is a $\bbQ$-generic filter over $\bold V[G_0][G(\alpha)]$ such that $\nu_0\in H$, and $E$ is the set of $\alpha\in C$ such that $H\cap A\neq\emptyset$ for every maximal antichain $A$ of $\bbQ$ such that $A\cap\alpha\in\overline{M}_\alpha$, then $E$ is a club of $\omega_1$ and $E\cap T^{\vec M}_S=\emptyset$.

But using (1), a standard density argument employing the current version of Lemma \ref{properness} shows that $T^{\vec M}_S$ remains stationary in $\bold V[G]$. Indeed, given a $\bbP_\kappa$-name $\name C\in\cH(\kappa)$ for a club of $\omega_1$, we may find working in $\bold V[G_0]$ a sufficiently correct $N\in S$ such that $\name C\in N$.  Letting then $q\in\bbP_\kappa$ be a condition witnessing $N\in S$, by the current version of Lemma \ref{properness}, the definition of $\bbQ^\ast$, and the way we construct generic conditions for this forcing, it follows that $q$ forces over $\bold V[G_0]$ that $\delta_N\in\name C\cap T^{\vec M}_S$. Hence, if $H$ is in fact generic over $\bold V[G]$, then we get a contradiction since forcing with $\bbQ$ over $\bold V[G]$ supposedly preserved stationary subsets of $\omega_1$.
\end{proof}

Lemma \ref{finallemmathm3} concludes the proof of Lemma \ref{thm3}.

We will finish the paper with a question regarding a potential strengthening of our principle $\CH$--$\omega_1$\text{-linked-Symm-}$\BPFA(\omega_2)$.

We first fix some notation. Let us say that an inner model $M$ is \emph{$\NS_{\omega_1}$-correct} iff for every $S\in\mathcal P(\omega_1)^M$, if $S$ is stationary in $M$, then $S$ is stationary in $\bold V$.

Let \emph{Local$^{++}$ $\CH$} be the statement that every set in $H(\omega_2)$ is in some $\NS_{\omega_1}$-correct ground model satisfying $\CH$.\footnote{As observed by Lietz, Local$^{++}$ $\CH$ fails if $\NS_{\omega_1}$ is $\aleph_1$-dense. Indeed, if $\{S_i\,:\,i<\omega_1\}$ is dense in $\mtcl P(\omega_1)\setminus\NS_{\omega_1}$, then no $\NS_{\omega_1}$-correct inner model containing $\{S_i\,:\,i<\omega_1\}$ together with a sequence $(e_\alpha\,:\,\omega<\alpha<\omega_1)$ of surjections $e_\alpha:\omega\to\alpha$ can satisfy $\CH$ since such an inner model would think that $\NS_{\omega_1}$ is $\aleph_1$-dense and the $\aleph_1$-density of $\NS_{\omega_1}$ implies $\lnot\CH$.}

Let us also define $\CH$--$\omega_1$\text{-linked-Symm-}$\BPFA(\omega_2)^{++}$ to be the variant of the principle $\CH$--$\omega_1$\text{-linked-Symm-}$\BPFA(\omega_2)$ obtained as the conjunction of the following two statements.

\begin{enumerate}
\item  Local$^{++}$ $\CH$
\item  Let $a\in H(\omega_2)$ and let $\varphi(x, y)$ be a $\Sigma_0$ formula in the language of $(H(\omega_2); \in, \NS_{\omega_1})$. Suppose for every $\NS_{\omega_1}$-correct ground model $M$, if $a\in M$ and $M\models\CH$, then it holds in $M$ that there is an $\omega_1$-linked symmetrically proper forcing notion $\mtcl Q\sub H(\omega_2)^M$ such that $\mtcl Q$ forces $(H(\omega_2); \in, \NS_{\omega_1})\models \exists y\varphi(a, y)$. Then $$(H(\omega_2); \in, \NS_{\omega_1})\models\exists y\varphi(a, y)$$
\end{enumerate}

We have the following.

\begin{proposition}
$\CH$--$\omega_1$\text{-linked-Symm-}$\BPFA(\omega_2)^{++}$ implies $\MM(\omega_1)$.
\end{proposition}

\begin{proof}
Let as assume $\CH$--$\omega_1$\text{-linked-Symm-}$\BPFA(\omega_2)^{++}$, let $\mathbb P$ be a poset on $\omega_1$ preserving stationary subsets of $\omega_1$, and let $\mtcl D$ be an $\omega_1$-sequence of dense subsets of $\mathbb P$. Let $\psi$ be a $\Sigma_1$ sentence in the language of $(H(\omega_2); \in, \NS_{\omega_1})$ with $\bbP$ and $\mtcl D$ as parameters expressing that there is a filter of $\bbP$ meeting all members of $\mtcl D$ or there is a stationary subset of $\omega_1$ which is destroyed by $\bbP$. 

Let $M$ be an $\NS_{\omega_1}$-correct ground model satisfying $\CH$ and containing $\mathbb P$ and $\mtcl D$. By $\NS_{\omega_1}$-correctness of $M$ and the fact that forcing with $\bbP$ over $\bold V$ preserves stationary subsets of $\omega_1$, we have that the same is true in $M$. If $\bbP$ is proper in $M$, then $\bbP$ is, in $M$, isomorphic to a P\v{r}\'{i}kr\'{y}-type forcing with stems in $\omega_1$ and hence it is an $\omega_1$-linked symmetrically proper forcing. And of course every $\bbP$-generic filter over $M$ meets all members of $\mtcl D$.  If, on the other hand, $\mtcl D$ is not proper in $M$, then by the fact that forcing with $\bbP$ over $M$ preserves stationary subsets of $\omega_1$ together with Proposition \ref{main-prop}, the paragraph right after this proposition, and Lemma \ref{CH2}, it follows that $(\bbQ^*)^M$ is, in $M$, an $\omega_1$-linked and symmetrically proper included in $\cH(\omega_2)$ and forcing the existence of a stationary subset of $\omega_1$ which is destroyed after forcing with $\bbP$.

It thus follows that in $M$ there is an $\omega_1$-linked symmetrically proper poset forcing $\psi$. Hence, by an appropriate instance of clause (2) in the definition of our principle $\CH$--$\omega_1$\text{-linked-Symm-}$\BPFA(\omega_2)^{++}$ we have that $\bold{V}\models \psi$. By our assumption that $\bbP$ preserves stationary subsets of $\omega_1$ it then has to be the case, in $\bold V$, that there is a filter of $\bbP$ meeting all members of $\mtcl D$. 
\end{proof}

Our question is the following.

\begin{question}\label{q1}
Is $\CH$--$\omega_1$\text{-linked-Symm-}$\BPFA(\omega_2)^{++}$ consistent?
\end{question}

We do not know if our construction for Theorem \ref{thm2} is such that every relevant inner model $\bold V[G\cap\bbP_\alpha\restriction\,\mtcl U^\alpha][H_\alpha]$, where $G$ is $\bbP_\kappa$-generic and $H_\alpha$ is the generic filter for $(\name{\bbQ}_\alpha)_{G\cap\bbP_\alpha\restriction\,\mtcl U^\alpha}$ given by $G$, is $\NS_{\omega_1}$-correct in $\bold V[G]$ -- or can be modified so as to produce such a construction. A positive answer would of course yield a positive answer to Question \ref{q1}.


\end{document}